\documentclass{amsart}
\usepackage[margin=1in]{geometry}
\usepackage{amsmath}
\usepackage{amssymb}
\usepackage{amsthm}
\usepackage{amscd}
\usepackage{enumerate}

\usepackage{graphicx}
\usepackage{amsmath,amsthm,amsfonts,amscd,amssymb,comment,eucal,latexsym,mathrsfs}
\usepackage{stmaryrd}
\usepackage[all]{xy}

\usepackage{epsfig}

\usepackage[all]{xy}
\xyoption{poly}
\usepackage{fancyhdr}
\usepackage{wrapfig}
\usepackage{epsfig}

\theoremstyle{plain}
\newtheorem{thm}{Theorem}[section]
\newtheorem{prop}[thm]{Proposition}
\newtheorem{lem}[thm]{Lemma}
\newtheorem{cor}[thm]{Corollary}
\newtheorem{conj}{Conjecture}

\theoremstyle{definition}
\newtheorem{defn}[thm]{Definition}
\theoremstyle{remark}

\newtheorem{example}{Example}

  \def\C{{\mathbb{C}}}   \def\F{{\mathbb{F}}}        \def\N{{\mathbb{N}}}   \def\Q{{\mathbb{Q}}} \def\R{{\mathbb{R}}}  \def\T{{\mathbb{T}}}      \def\Z{{\mathbb{Z}}}

\newcommand\vre{\varepsilon}
\newcommand\varpesilon{\vre}
\newcommand\varespilon{\vre}

\newcommand\Bor{{\operatorname{Bor}}}

\newcommand\Det{\operatorname{det}}

\newcommand\id{\operatorname{id}}

\newcommand\Meas{{\operatorname{Meas}}}

\newcommand\Prob{\operatorname{Prob}}

\renewcommand\Re{\operatorname{Re}}

\newcommand\supp{\operatorname{supp}}

\newcommand\Span{\operatorname{span}}

\newcommand{\actson}{\curvearrowright}
\newcommand{\actons}{\curvearrowright}

\newcommand{\ip}[1]{\langle #1 \rangle}

\begin{document}
\title{Harmonic Models and Bernoullicity}      
\author{Ben Hayes}\thanks{The author gratefully acknowledges support from  NSF Grant DMS-1827376.}
\address{University of Virginia\\
          Charlottesville, VA 22904}
\email{brh5c@virginia.edu}
\date{\today}
\maketitle

\begin{abstract}
We give many examples of algebraic actions which are factors of Bernoulli shifts. These include certain harmonic models over left orderable groups of large enough growth, as well as algebraic actions associated to certain lopsided elements in any left orderable group. For many of our examples, the acting group is amenable so these actions are Bernoulli (and not just a factor of a Bernoulli), but there is no obvious Bernoulli partition.
\end{abstract}

\tableofcontents

\section{Introduction}

The goal of this paper is to give many examples of algebraic actions which are either Bernoulli shifts, or factors of Bernoulli shifts. Given a countable, discrete, group $G$ and a probability space $(B,\beta)$ the \emph{Bernoulli shift with base $(B,\beta)$} is the action $G\actson (B^{G},\beta^{\otimes G})$ given by $(gb)(h)=b(g^{-1}h)$ for $h,g\in G,b\in B^{G}.$  Bernoulli shifts have been a natural class of action of interest since the beginning of ergodic theory. In many ways, this is because they are the most natural example of a probability measure-preserving action of a group and are in some sense an ergodic theoretic version of the action of a group on itself.

 Bernoulli shifts are also inherently tied with dynamical entropy: the first application of the Kolmogorov-Sina\v\i\ entropy was to show that Bernoulli shifts with different base entropy are not isomorphic (here the base entropy is $-\sum_{b\in B}\beta(\{b\})\log\beta(\{b\})$). Work of Ornstein \cite{OrnClassify1, OrnClassify2} then showed that the class of Bernoulli shifts actions over $\Z$ are completely classified by dynamical entropy. This was extended to the amenable case by Ornstein-Weiss in \cite{OrnWeiss}. A striking combination of recent results shows that the same is true for the class of \emph{sofic} groups: Bowen showed in \cite{Bow} that two Bernoulli shifts with different base entropy are nonisomorphic when the acting group is sofic, and a recent result of Seward \cite{SewardOrn} (following up on work of Bowen in \cite{BowenOrn}) shows that  for \emph{any} group if two probability spaces have the same Shannon entropy, then the Bernoulli shifts with that base are isomorphic. It is not known if two Bernoulli shifts with different base entropies are not isomorphic for general groups.

Another significant aspect of the study of Bernoulli shifts is that, for amenable groups, we can often say many actions are Bernoulli even if they have no obvious generating partition with independent translates. This is due to \emph{Ornstein theory}, first developed by Ornstein for the case of $\Z$ \cite{OrnClassify1, OrnClassify2}, and then by Ornstein-Weiss for amenable groups \cite{OrnWeiss}. Of particular interest for us are \emph{algebraic actions:} these are actions of a countable, discrete, group $G$ by continuous automorphisms of a compact, metrizable group $X.$ If we give $X$ the Haar measure $m_{X},$ then $G\actson (X,m_{X})$ is a probability measure-preserving action. When $G=\Z^{d}$ there are many results that say that frequently algebraic actions are Bernoulli: for example, Katznelson showed in \cite{KatzErgAut} that ergodic automorphism of a finite-dimensional tours are Bernoulli,  Lind showed   that the same is true for an infinite torus \cite{LindErgAut}, and Rudolph-Schmidt showed \cite{RudolphSchmidtCPE} that  algebraic actions of $\Z^{d}$ with completely positive entropy are Bernoulli.  This can be used to efficiently ``detect" Bernoullicity of many natural algebraic actions of $\Z^{d}.$ Unfortunately, little is known about Bernoullicity of algebraic actions outside of the $\Z^{d}$ case, even in the amenable case. For example, it is still not known if an algebraic action of an amenable group with completely positive entropy is Bernoulli.

It is known that for non-amenable groups there can be factors of Bernoulli shifts which are not Bernoulli. One example is the Popa factor: $G\actson (\T^{G}/\T,m_{\T^{G}/\T})$ where $\T=\R/\Z$ and $\T$ is viewed inside of $\T^{G}$ as the set of elements with constant coordinates. If $G$ has Property (T), then Popa showed  (see \cite[Theorem 1]{PopaCohomologyOE}, see also \cite{PopaSasyk} for related results) that the Popa factor is not Bernoulli, and it is clearly an algebraic action. On the other hand, when $G$ is treeable, then the Popa factor is Bernoulli, by Gaboriau-Seward \cite{GabSewPopaFactor}. There are other nice examples of algebraic actions of free groups which are Bernoulli, and not obviously so, due to Lind-Schmidt in \cite{LindSchmidtFreeBern}. Outside of these, we do not known of many nonobvious examples of algebraic actions  that are even \emph{factors} of Bernoulli.

For readers who are less familiar with the nonamenable setting, let us remark that being a factor of Bernoulli still has significant consequences in the nonamenable case: if the acting group is sofic then the action has completely positive entropy \cite{KerrCPE},  it is solidly ergodic by the work of Chifan-Ioana \cite{SolidErg}, it is mixing of all orders, has spectral gap, etc. Because of this, we have several conditions which guarantee that an action is \emph{not} a factor of a Bernoulli shift when the acting group is nonamenable: if it is not mixing, if it is orbit equivalent to a compact action, if it is not strongly ergodic, if its Koopman representation does not embed into an infinite direct sum of the left regular representation, if it not solidly ergodic. For example, if $\pi$ is a unitary representation of a group $G,$ and if $\pi$ does not embed into an infinite direct sum of the left regular representation, then the corresponding Gaussian action is \emph{not} a factor of a Bernoulli shift.
 If the acting group is assumed sofic, we can also say that any action with zero sofic entropy with respect to some sofic approximation is not a factor of a Bernoulli shift. By a recent result of Bowen \cite{BowenGeneric}, this implies that a \emph{generic} action of a sofic group is not a factor of a Bernoulli shift. Special to the sofic case, we can also exhibit actions that are inverse limits of Bernoulli shifts but have zero entropy and are thus not factors of a Bernoulli shift (see \cite[Corollary 4.4]{BowenGeneric}). Moreover, if $H\leq G,$ and if $H\actson (Z,\zeta)$ is a probability measure-preserving action which is not a factor of a Bernoulli shift over $H,$ then the coinduction of $H$ to $G$ gives an action which is not a factor of a Bernoulli shift. From this and \cite{BowenGeneric,SchmidtSpectralGap} one can show that given if $G$ is any group with an infinite subgroup $H$ so that $H$ is either sofic, or does not have Property (T), then $G$ has an ergodic action which is not a factor of a Bernoulli shift. Such an action can automatically be made free by taking a product with a Bernoulli shift (this preserves ergodicity as well as the property of ``not a factor of a Bernoulli shift").

 In fact, in the nonamenable setting, it is actually harder to do the reverse: find criterion on  an ergodic action which guarantees that it does not \emph{factor onto} a Bernoulli shift. For example, no compact action can factor onto a Bernoulli shift (and in the nonamenable setting, even an action orbit equivalent to a compact action cannot factor onto a Bernoulli shift see \cite[Section 4.6.1]{BowenExamples}), but being not mixing, or having a Koopman representation which is not embeddable into the infinite direct sum of the left regular representation is not sufficient. When the acting group is sofic, and the action is free, factoring onto a Bernoulli shift is equivalent to having a factor with positive entropy with respect to some sofic approximation, by a recent stunning result of Seward \cite{SewardSinai}. So, in the sofic case, a free action does not factor onto a Bernoulli shift is equivalent to having \emph{completely zero entropy} (i.e. every action has zero entropy) with respect to some (equivalent any) sofic approximation.

The main result of this paper gives a large class of examples of algebraic actions which are factors of Bernoulli shifts. When the acting group is amenable, this implies that they \emph{are} Bernoulli shifts, by Ornstein theory. We need some preliminary notions from group theory. A \emph{left-invariant order} on a group $G$ is a total order $<$ on $G$ so that if $x,y\in G$ and $x<y,$ then $gx<gy$ for all $g\in G.$ Given such an order, an element $g\in G$ is \emph{positive} if $g>1.$ A group is \emph{left-orderable} if there is a left-invariant order on $G,$ and a \emph{left-ordered group} is a group equipped with a fixed left-invariant order. We refer the reader to Section \ref{S:applications} for a discussion of many examples of left-orderable groups, both amenable and nonamenable. Finally, if $G$ is finitely generated with finite generating set $S,$ then for an integer $R\geq 1$ we let $B_{R}(S)$ be the ball of radius $R$ centered at the identity in the word metric coming from $S.$ That is, $B_{R}(S)=(S\cup \{1\}\cup S^{-1})^{R}.$
Unlike the usual situation, we will typically be interested in \emph{anti-symmetric} generating sets $S,$ i.e. ones for which $S\cap S^{-1}=\varnothing.$ Lastly, given $f\in \Z(G),$ we let $X_{f}$ be the Pontryagin dual of $\Z(G)/\Z(G)f.$ That is, $X_{f}$ is the space of continuous homomorphisms from $\Z(G)/\Z(G)f$ to $\R/\Z.$   We now present the two main results of the paper.

\begin{thm}\label{T:harmonic model intro}
Let $G$ be a finitely generated, left-ordered group, and assume that $S$ is a set of  positive generators. Suppose $|B_{R}(S)|\geq CR^{d}$ for some constants $C>0,d\geq 5.$ Let $f=m+\sum_{s\in S}a_{s}s\in \Z(G)$ and assume that $a_{s}\ne 0$ for all $s\in S,$ and that $\sum_{s\in S}|a_{s}|=|m|.$ Then $G\actson (X_{f},m_{X_{f}})$ is a factor of a Bernoulli shift. If $G$ is assumed amenable, then $G\actson (X_{f},m_{X_{f}})$ is isomorphic to a Bernoulli shift with entropy $\log(m).$

\end{thm}

If in the above $m>0$ and $a_{s}<0$ for all $s\in S,$ then $G\actson (X_{f},m_{X_{f}})$ is called a \emph{harmonic model}. This is because we can write $f=m(1-x)$ where $x=\sum_{s\in S}\mu(s)s$ for some $\mu\in \Prob(G),$ and $X_{f}$ is in some sense the space of ``$\T$-valued $\mu^{*}$-harmonic functions" (here $\mu^{*}(g)=\mu(g^{-1})$), see \cite[Section 1]{BowenLi} for more details. The ergodic theory of the harmonic model was previously studied by several authors, see e.g.  \cite{BowenLi, KitchensSchmidt, LindSchmidt1, LindSchmidt2, Sandpiles}. The proof of Theorem \ref{T:harmonic model intro} uses our results in \cite{MeMaxMinWC}, which allow one to measurably extend the convolution map $\{-n,\cdots,n\}^{G}\to \R^{G}$ from convolving with $\ell^{1}$-vectors to convolving with $\ell^{2}$-vectors. As explained in \cite[Proposition 3.8]{MeMaxMinWC}, there is no canonical way to measurably extend the convolution map $\{-n,\cdots,n\}^{G}\to \R^{G}$ to case of convolving with $\ell^{p}$-vectors, with $p>2.$ Because of this, one cannot use the same techniques we used to prove Theorem \ref{T:harmonic model intro} to weaken the growth assumption on $G$ (see the remarks following Corollary \ref{C:well balanced superpoly growth inverses} for more information).

If  $f\in \C(G)$ and $f=m+\sum_{s\in S}a_{s}s$ and $\sum_{s\in S}|a_{s}|<|m|,$ then $f$ is called \emph{lopsided}. In this case, we can drop the assumption on the growth rate of $G,$ and only require orderability.

\begin{thm}\label{T:lopsided intro}
Let $G$ be a finitely generated, left-ordered group, and assume that $S$ is a set of positive generators, and that $|S|\geq 2.$ Let $f=m+\sum_{s\in S}a_{s}s\in\Z(G)$ and assume that $\sum_{s\in S}|a_{s}|<|m|.$ Then $G\actson (X_{f},m_{X_{f}})$ is a factor of a Bernoulli shift. If $G$ is assumed amenable, then $G\actson (X_{f},m_{X_{f}})$ is isomorphic to a Bernoulli shift with entropy $\log(m).$
\end{thm}

We remark here that there are \emph{many} examples of left-orderable groups including: torsion-free nilpotent groups, polycyclic groups, certain groups of intermediate growth, Thompson's group, free groups, certain mapping class groups etc. See Section \ref{S:applications} for detailed examples with references.

In each of Theorem \ref{T:harmonic model intro}, \ref{T:lopsided intro}, if $G$  is amenable, then the reason we know that $G\actson (X_{f},m_{X_{f}})$ is a Bernoulli shift by Ornstein theory.  We know \emph{which} Bernoulli shift it is by of the results of \cite{Den, DenSchmidt, LiThom}, and the fact that we can directly compute the Fuglede-Kadison determinant (see Appendix \ref{S:FKD}) in this case. If $G$ is not assumed amenable, then by \cite{Me5} (see also \cite{BowenLi} in the harmonic model case, and \cite{BowenEntropy, KLi} in the expansive case), and Appendix \ref{S:FKD}, we know that the entropy of $G\actson (X_{f},m_{X_{f}})$ has entropy $\log(m).$ Unfortunately, Ornstein theory is not developed in the nonamenable case, so we do not know if $G\actson (X_{f},m_{X_{f}})$ is isomorphic to a  Bernoulli shift. However, if $m$ is odd, the factor map we use to show that $G\actson (X_{f},m_{X_{f}})$ is a factor of a Bernoulli is a map between spaces of equal entropy, and we suspect that it is injective modulo null sets. This is known in certain examples when $G$ is the free group by work of Lind-Schmidt \cite{LindSchmidtFreeBern}.

We mention that  in Theorems \ref{T:harmonic model intro} and \ref{T:lopsided intro} we do not actually need the group to be totally ordered. A \emph{left-invariant partial order} on $G$ is a partial order $\preceq$ so that if $x,y\in G$ and $x\preceq y,$ then $gx\preceq gy$ for all $g\in G.$ If we set $P=\{x\in G:x\succ 1\},$ then:
\begin{itemize}
\item $x,y\in P$ implies $xy\in P,$
\item $P\cap P^{-1}=\varnothing.$
\end{itemize}
Equivalently, $P$ is a subsemigroup of $G$ with $1\notin P$ (we remark that left-invariant partial orders on groups also appeared in \cite{AlpMeyRyu} but for different reasons).
A subset $P$ of $G$ satisfying the above two axioms is called a \emph{positive semigroup}. If we are given a positive semigroup, then we can define a left-invariant partial order on $G$ by $x\preceq y$ if $x^{-1}y\in P.$ So positive semigroups correspond to left-invariant partial orders on $G.$ We can extend Theorems \ref{T:harmonic model intro} and \ref{T:lopsided intro} to groups $G$ so that there is a positive semigroup $P$ with $\ip{P}=G,$ see Section \ref{S:orderability}. Such groups cannot be torsion, but we also have examples of such groups which are not torsion-free. See Section \ref{S:applications} for a discussion of examples.

We finish by discussing the organization of the paper. In Section \ref{S:background} we discuss some background results for the paper. These involve the technique we used in \cite{MeMaxMinWC} to measurably extend the convolution map $\{-n,\cdots,n\}^{G}\to \R^{G}$ from convolving with $\ell^{1}$-vectors to convolving with $\ell^{2}$-vectors. We state the main results on this construction obtained in \cite{MeMaxMinWC}, which are the main tool we will use to get factor maps from Bernoulli shifts. In Section \ref{S:growth} we explain how the growth rate assumption on $G$ shows up, this is related to decay rates of return time probability of random walks on $G.$ In Section \ref{S:orderability}, we explain why the orderability assumption on  $G$ is relevant. We also prove the two main results of the paper in this section. In Section \ref{S:applications}, we give many examples of actions we can prove are factors of Bernoulli shifts using our work. We split this into the amenable case and the nonamenable case, since in the amenable case we get that they are isomorphic to Bernoulli shifts as a consequence of Ornstein theory. In section \ref{S:closing} we give some closing remarks, as well as state some conjectures related to our work. In particular, we strongly suspect that the factor maps we produce in the nonamenable case are often isomorphisms. Appendix \ref{S:Tracial vNa} gives some background results on tracial von Neumann algebras we will use in the paper. In particular, in Section \ref{S:growth} we require a few background lemmas whose proof we give in Appendix \ref{S:nclp},\ref{S:general L2 inverses}. Appendix \ref{S:general L2 inverses} contains general results on $L^{2}$ formal inverses which may be of independent interest. We will need to compute the entropy of the algebraic actions in question using the results of \cite{LiThom,Me5}. This requires computing some Fuglede-Kadison determinants, which we do in Appendix \ref{S:FKD}. Lastly, the reader may be more familiar with arguments involving lopsided elements and $\ell^{1}$ inverses, or even inverses in the group von Neumann algebra, as opposed to $\ell^{2}$ formal inverses. We discuss the difference between these notions in Appendix \ref{S:inverses}.

\textbf{Acknowledgments.}
I thank Doug Lind for interesting discussions related to this work.
I thank Yago Antol\'{i}n, Thomas Koberda, and Yash Lodha and  for useful discussions related to left-orderable groups. I thank Lewis Bowen and Klaus Schmidt for their  comments on an earlier version of the paper.

\subsection{Conventions and Notation}
If $(X,\mu)$ is a measure space and $K$ is a compact Hausdorff space, we let  $\Meas(X,K)$ be the space of all measurable maps $X\to K,$ where two maps are identified if they agree almost everywhere. We give $\Meas(X,K)$ the topology of convergence in measure: so a basic  neighborhood of $\Theta\in \Meas(X,K)$ is given by
\[U_{V,\varepsilon}(\Theta)=\{\Psi\in \Meas(X,K):\mu(\{x:(\Psi(x),\Theta(x))\in V\})>1-\varepsilon\},\]
where $V$ is a neighborhood of the diagonal in $X\times X$ and $\varpesilon>0.$ We often call this topology the \emph{measure topology}. If $G$ is a countable, discrete group and $G\actson (X,\mu)$ is probability measure-preserving, and $G\actson K$ by homeomorphisms, we let $\Meas_{G}(X,K)$ be the set of (almost surely) $G$-equivariant elements of $\Meas(X,K).$ If $A$ is a set, we let $G\actson A^{G}$ be given by $(gx)(h)=x(g^{-1}h),$ for $x\in A^{G},g,h\in G.$
If $A$ is a compact, Hausdorff space, then so is $A^{G}$ and this action is by homeomorphisms. If $E$ is a finite set, then we equip $E$ with the uniform probability measure $u_{E}.$ If $Y$ is a locally compact, Hausdorff space, we let $\Prob(Y)$ be the  space of all Radon probability measures on $Y.$

If $G$ is a countable, discrete group we let $\C(G)$ denote its complex group ring. Recall that this is the ring of all formal sums $\sum_{g}a_{g}g$ where $a_{g}\in \C$ and all but finitely many of the $a_{g}$ are $0.$ We let $c_{c}(G)$ be all finitely supported functions $f\colon G\to \C,$ and $c_{0}(G)$ all functions $f\colon G\to \C$ so that $\{g:|f(g)|>\varepsilon\}$ is finite for every $\varpesilon>0.$ We define $\tau\colon \C(G)\to \C$ by $\tau(\sum_{g}a_{g}g)=a_{1}.$ Given $\alpha\in \C(G),$ we define $\widehat{\alpha}\in c_{c}(G)$ by
\[\widehat{\alpha}(g)=\tau(g^{-1}\alpha).\]
We will adopt obvious notation such as $c_{c}(G,\R),c_{0}(G,\R),\ell^{p}(G,\R)$ to denote $c_{c}(G)\cap \R^{G},c_{0}(G)\cap \R^{G},\ell^{p}(G)\cap\R^{G}$ etc. Similar remarks apply to $\R(G),\Q(G),\Z(G)$ etc.
For $\alpha\in \C(G),$ let $\|\alpha\|_{2}$ be given by $\|\alpha\|_{2}=\tau(\alpha^{*}\alpha).$
For $\alpha=\sum_{g}\alpha_{g}g\in \C(G),$ we let
\[\alpha^{*}=\sum_{g}\overline{\alpha_{g^{-1}}}g.\]
Given $\xi\in \C^{G},\alpha\in \C(G),$ we define $\alpha\xi\in \C^{G}$ by
\[(\alpha\xi)(g)=\sum_{h}\widehat{\alpha}(h)\xi(h^{-1}g).\]
Similarly, we define $\xi\alpha$ by
\[(\xi\alpha)(g)=\sum_{h}\xi(gh^{-1})\widehat{\alpha}(h).\]
Related to the above, we introduce the following notation. If $f\in \C(G),$ we let $\lambda(f)\colon \ell^{2}(G)\to \ell^{2}(G)$ be defined by
\[\lambda(f)\xi=f\xi.\]
If $\xi\in \C^{G},$ we let $\supp(\xi)=\{g\in G:\xi(g)\ne 0\}.$

Let $A$ be a locally compact, abelian group. We let $\widehat{A}$ be the set of continuous homomorphisms $\chi\colon A\to \T,$ where $\T=\R/\Z.$   For $\mu\in \Prob(A),$ we define the \emph{Fourier transform of $\mu,$} $\widehat{\mu}\colon \widehat{A}\to \C$ by
\[\widehat{\mu}(\chi)=\int_{A}\exp(2\pi i \chi(x))\,d\mu(x).\]
If $G$ is a countable, discrete group we identify $(\T^{G})^{\widehat{}}$ with $\Z(G)$ under the pairing
\[\ip{\theta,\alpha}_{\T}=\sum_{g\in G}\alpha_{g}\theta(g),\mbox{ for $\theta\in \T^{G},\alpha=\sum_{g\in G}\alpha_{g}g\in \Z(G).$}\]
For $f\in \C^{G},$ we let $f^{*}\in \C^{G}$ be given by $(f^{*})(g)=\overline{f(g^{-1})}.$

 \section{Background results}\label{S:background}
In \cite[Section 3]{MeMaxMinWC}, we defined a way  to ``measurably" extend the map $\R^{G}\to \T^{G}$ given by convolution by a finitely supported vector to the case of convolving by an $\ell^{2}$-vector. We restate the results here for the convenience of the reader.

\begin{thm}\label{T: background convolution}
Let $G$ be a countable, discrete group. Fix a $\nu\in \Prob(\R)$ with mean zero and finite second moment. There is a unique map $\ell^{2}(G,\R)\to \Meas(\R^{G},\nu^{\otimes G}, \T^{G}),$  $\xi\mapsto \Theta_{\xi}$ so that:
\begin{enumerate}[(i)]
\item $\Theta_{\xi}(x)(g)=(x\xi^{*})(g)+\Z$ \mbox{ for all $\xi\in c_{c}(G,\R),$ and all $x\in \R^{G}$, $g\in G.$}
\item $\xi\mapsto \Theta_{\xi}$ is continuous if we give $\ell^{2}(G,\R)$ the $\|\cdot\|_{2}$-topology and $\Meas(\Z^{G},\nu^{\otimes G}, \T^{G})$ the topology of convergence in measure.
\end{enumerate}
Moreover, if we set $\mu_{\xi}=(\Theta_{\xi})_{*}(\nu^{\otimes G}),$ then
\[\widehat{\mu}_{\xi}(\alpha)=\prod_{g\in G}\widehat{\nu}((\alpha\xi)(g))\]
with the product on the right hand side converging absolutely.
\end{thm}

We refer the reader to \cite[Proposition 3.8]{MeMaxMinWC} for a discussion of the fact that $2$ is \emph{optimal} in the above theorem. Namely, if $p>2,$ then as long as $\nu$ is not the point mass at $0,$  there does not exist a map $\ell^{p}(G,\R)\to \Meas(\R^{G},\nu^{\otimes G},\T^{G})$ which is continuous if we give $\ell^{p}(G,\R)$ the $\|\cdot\|_{p}$-topology and $\Meas(\R^{G},\nu^{\otimes G},\T^{G})$ the topology of convergence in measure, and which agrees with $x\mapsto (g\mapsto (x*\xi)(g)+\Z))$ when $\xi\in c_{c}(G,\R).$

As we mention later (see Section \ref{S:closing}), Theorem \ref{T: background convolution} is significantly easier when $\xi\in \ell^{1}(G,\R),$ and in that case we do not need to assume that $\nu$ has mean zero. However, in order to apply this to the context of $G\actson (X_{f},m_{X_{f}})$ for $f\in \Z(G),$ this would force $f$ to be invertible in the convolution algebra $\ell^{1}(G).$ As we discuss in Appendix \ref{S:inverses}, the $\ell^{1}$ version of Theorem \ref{T: background convolution} is insufficient for our purposes. The reader may also be familiar that in previous works (see \cite{Li, Me7, Me13} for example) one assumed that $f$ is invertible in the group von Neumann algebra. We remark in Appendix \ref{S:inverses} that this would force us to have the acting group be nonamenable, in general. Since we do \emph{not} want to restrict ourselves to the nonamenable case, and want to study both the amenable and nonamenable setting, we want to work with the $\ell^{2}$  version of Theorem \ref{T: background convolution}.

\begin{defn}
Let $G$ be a countable, discrete group, and $f=\sum_{g\in G}f_{g}g\in \C(G).$ We say $f$ is \emph{semi-lopsided} if
\begin{itemize}
    \item $f_{1}> 0$
    \item $\sum_{g\in G:g\ne 1}|f_{g}|\leq f_{1}.$
\end{itemize}
We say $f$ is \emph{lopsided} if $f_{1}>\sum_{g\in G\setminus \{1\}}|f_{g}|.$
We say that $f$ is \emph{well-balanced} if $f_{g}\leq 0$ for all $g\in G\setminus \{1\},$  $f_{1}>0,$ and $\sum_{g\in G}f_{g}=0.$
\end{defn}
Some authors use \emph{lopsided} to mean $|f_{1}|>\sum_{g\in G\setminus\{1\}}|f_{g}|.$ However, we are primarily interested in the case $f\in\Z(G)$ and the corresponding action $G\actson (X_{f},m_{X_{f}}).$  Since $X_{f}=X_{-f}$ we may take $f_{1}> 0$ without loss of generality.

We may regard $\Prob(G)$ as all $\mu\in \ell^{1}(G)$ so that $\mu(g)\geq 0$ for all $g$ and so that $\|\mu\|_{1}=1.$
An equivalent way to say that $f$ is well-balanced is that it can be written as $f=m(1-p)$ where $m\in\Z,$ and $\widehat{p}\in \Prob(G).$
We will apply Theorem \ref{T: background convolution} to show that, in many cases, a semi-lopsided element $f\in\Z(G)$ gives rise to  an algebraic action $G\actson (X_{f},m_{X_{f}})$ which is a factor of a Bernoulli shift. In order to do this, we need find an element $\xi\in \ell^{2}(G,\R)$ to apply Theorem \ref{T: background convolution} to. For this, we will need a generalized notion of invertibility.

\begin{defn}
Let $f\in \Z(G),$ we say that $\alpha\in \R^{G}$ is a \emph{formal right  inverse} of $f$ if $\alpha f=\delta_{1}.$ If $\alpha\in c_{0}(G),$ we will say that $\alpha$ is a \emph{$c_{0}$  formal right inverse of $f.$} If $\alpha\in \ell^{p}(G),$ we will say that $\alpha$ is an \emph{$\ell^{p}$ formal right  inverse}.
\end{defn}
In the above definition, it can be shown that if $p=2,$ and $\alpha$ is an \emph{$\ell^{2}$ formal right  inverse of $f$}, then $f\alpha=\delta_{1}$ (see e.g. \cite[Proposition 2.2]{MeWE}). Thus, if $1\leq p\leq 2$ will simply say that $f$ has an $\ell^{p}$ formal inverse.  The notion of an $\ell^{2}$ formal inverse will be the main way in which we obtain $\ell^{2}$ vectors to apply Theorem \ref{T: background convolution} to. As we mentioned before, the case $p=2$ is optimal in Theorem \ref{T: background convolution}, so we cannot apply our methods to $f\in\Z(G)$ if we only assume that $f$ has an $\ell^{p}$ formal inverse for $p>2,$ or which has a $c_{0}$ formal inverse. One reason why $\ell^{2}$ formal inverses are helpful is the following Corollary of Theorem \ref{T: background convolution}. This corollary is shown explicitly in \cite[Section 3]{MeMaxMinWC}, but we state it here for convenience.

\begin{cor}
Let $G$ be a countable, discrete group fix a $\nu\in \Prob(\R)$ with mean zero and finite second moment. For $\xi\in \ell^{2}(G,\R),$ let $\Theta_{\xi}$ be defined as in Theorem \ref{T: background convolution} for this $\nu.$ Suppose that $f\in \Z(G)$ has an $\ell^{2}$ formal inverse $\xi.$ Then $(\Theta_{\xi})_{*}(\nu^{\otimes G})$ is supported on $X_{f}.$

\end{cor}

\section{Proof of the main Theorem}

We will use Theorem \ref{T: background convolution}  to prove that for certain choices of $G$ and for semi-lopsided $f\in \Z(G),$ we have that $G\actson (X_{f},m_{X_{f}})$ is a factor of a Bernoulli shift.  If $\xi$ is an $\ell^{2}$ inverse to $f,$ then $\mu_{\xi}=(\Theta_{\xi})_{*}(\nu^{\otimes G})$ is a probability measure on $X_{f}$ for every $\nu\in \Prob(\Z)$ with mean zero and a finite second moment. By definition, $G\actson (X_{f},\mu_{\xi})$ is a factor of a Bernoulli shift, and so we just want to force $\mu_{\xi}$ to be $m_{X_{f}}.$ We do this by computing its Fourier transform using Theorem \ref{T: background convolution} and verifying that it agrees with the Fourier transform of $m_{X_{f}}.$

In summary, what we want to find are classes of countable discrete groups $G,$   semi-lopsided  $f\in \Z(G),$ and  probability measures $\nu\in\Prob(\Z)$ with mean zero and finite second moment, so that:
\begin{itemize}
\item $f$ has an $\ell^{2}$ formal inverse $\xi,$
\item  if we set $\mu_{\xi}=(\Theta_{\xi})_{*}(\nu^{\otimes G}),$ then we can use the Fourier transform formula to show that $\widehat{\mu}_{\xi}=1_{\Z(G)f},$ and thus that $\mu_{\xi}=m_{X_{f}}.$
\end{itemize}

These two bulleted items are where the growth assumption on $G,$ and where the orderability of $G$ appear, respectively. We explore these in the next two subsections.

\subsection{$\ell^{2}$ formal inverses and growth}\label{S:growth}
In this section, we concentrate on conditions which guarantee that $f$ has a $\ell^{2}$ formal inverse. If $f$ is lopsided, then by standard Banach algebra arguments it has an $\ell^{1}$ formal inverse. If $f$ is semi-lopsided, but not lopsided, then it can be written as $f=m(1-x)$ with $m\in\Z,x\in \Q(G)$ and $|\widehat{x}|\in \Prob(G).$ So we focus on conditions that guarantee that if $1-x\in \R(G)$ with $|\widehat{x}|\in \Prob(G),$ then $(1-x)$ has a $\ell^{2}$  formal inverse. Formally, one considers the geometric series $(1-x)^{-1}=\sum_{n}x^{n}$ and attempts to prove that this converges in $\ell^{2}.$  The following two lemmas will be helpful in this regard. The proofs we give of them utilize the machinery of \emph{tracial von Neumann algebras}, which require us to recall a few preliminaries. To ensure that the core ideas of this section remain in the foreground, we have relegated the proof of these two lemmas to Appendix \ref{S:Tracial vNa}.


\begin{lem}\label{L:sot convergence}
Suppose that $x\in \C(G)$ and $|\widehat{x}|\in \Prob(G).$ Lastly suppose that the group generated by $\{a^{-1}b:a,b\in \supp(\widehat{x})\}$ is infinite. Then
 $\|x^{n}\xi\|_{2}\to_{n\to\infty}0$ for all $\xi\in \ell^{2}(G).$

\end{lem}

\begin{lem}\label{C:reductoin to positive case}
Let $G$ be a countable discrete group, and $f\in \C(G)$ with $f=1-x$ with $|\widehat{x}|\in \Prob(G).$ Then $f$ has an $\ell^{2}$ formal inverse if  $1-\left(\frac{x^{*}+x}{2}\right)$ has an $\ell^{2}$ formal inverse.

\end{lem}

As an application of Lemma \ref{L:sot convergence}, in the well-balanced case we can completely characterize when $1-x$ has an $\ell^{2}$ formal inverse as well as compute what this inverse has to be.

\begin{lem}\label{L:what is the inverse}
Suppose that $x\in \C(G)$ and $|\widehat{x}|\in \Prob(G\setminus\{1\}).$ Assume that the group generated by $\{a^{-1}b:a,b\in \supp(\widehat{x})\}$ is infinite. Let $y\in \R(G)$ with $\widehat{y}=|\widehat{x}|.$
\begin{enumerate}[(i)]
\item We have that $1-x$ has an $\ell^{2}$ formal inverse if and only if the series 
\[\sum_{n=0}^{\infty}x^{n}\delta_{1}\]
converges conditionally. Moreover, if $1-x$ has an $\ell^{2}$ formal inverse, then this $\ell^{2}$ formal inverse is $\sum_{n=0}^{\infty}x^{n}\delta_{1}$.
\label{item:conditonally convergent inverse}
\item  A sufficient condition for  $1-x$ to have an $\ell^{2}$ formal inverse is that $\sum_{n,m\geq 0}\tau((y^{*})^{m}y^{n})<\infty.$ If $\widehat{x}\geq 0,$  then this sufficient condition is also necessary. \label{item:almost iff inverse}

\end{enumerate}

\end{lem}

\begin{proof}
For a natural number $N\geq 1,$ set $\xi_{N}=\sum_{n=0}^{N}x^{n}.$

(\ref{item:conditonally convergent inverse}): Suppose $1-x$ has an $\ell^{2}$ inverse $\xi.$  Then
\[\|\xi_{N}\delta_{1}-\xi\|_{2}=\|\xi_{N}(1-x)\xi-\xi\|_{2}=\|x^{N+1}\xi\|_{2}\to_{N\to\infty}0,\]
the last step following from  Lemma \ref{L:sot convergence}.

Conversely, suppose that $\sum_{n=0}^{\infty}x^{n}\delta_{1}$ converges conditionally, and set $\xi=\sum_{n=0}^{\infty}x^{n}\delta_{1}.$ Then
\[(1-x)\xi=\lim_{N\to\infty}(1-x)\xi_{N}=\lim_{N\to\infty}\delta_{1}-x^{N+1}\delta_{1}=\delta_{1},\]
by Lemma \ref{L:sot convergence}.

(\ref{item:almost iff inverse}): First, suppose that $\sum_{n,m\geq 0}\tau((y^{*})^{m}y^{n})<\infty.$ For $N\geq M\geq 1,$ we have that
\[\|\xi_{N}-\xi_{M}\|_{2}^{2}=\sum_{M\leq n,m\leq N}\tau((x^{*})^{m}x^{n})\leq \sum_{M\leq n,m\leq N}\tau((y^{*})^{m}y^{n}).\]
So
\[\lim_{M\to\infty}\sup_{N\geq M}\|\xi_{N}-\xi_{M}\|_{2}=0,\]
by the dominated convergence theorem. Hence $\xi_{N}$ is a Cauchy sequence and thus converges. Hence, by part (\ref{item:conditonally convergent inverse}) we know that $1-x$ has an $\ell^{2}$ formal inverse.

Conversely, suppose that $\xi$ is an $\ell^{2}$ formal inverse to $1-x$ and that $\widehat{x}\geq 0.$ Let $\xi_{N}$ be defined as in the first half of the proof. Then
\[\|\xi_{N}\delta_{1}-\xi\|_{2}\to 0.\]
Hence,
\[\|\xi\|_{2}^{2}=\lim_{N\to\infty}\|\xi_{N}\|_{2}^{2}=\lim_{N\to\infty}\sum_{0\leq n,m\leq N}\tau((x^{*})^{m}x^{n})=\sum_{n,m\geq 0}\tau((x^{*})^{m}x^{n}).\]
So
\[\sum_{n,m\geq 0}\tau((x^{*})^{m}x^{n})<\infty.\]

\end{proof}

The above result is much easier to say when $\widehat{x}\geq 0$ and $x$ is self-adjoint.

\begin{lem}\label{L:easier in sa case}
Let  $f\in \C(G)$ be of the form $f=1-x$ with $\widehat{x}\in \Prob(G).$ Suppose that $x=x^{*},$ and that $\ip{\{ab:a,b\in \supp(\widehat{x})\}}$ is infinite. Then $f$ has an $\ell^{2}$ formal inverse if and only if $\sum_{k}k\tau(x^{k})<\infty.$

\end{lem}

\begin{proof}
By Lemma \ref{L:what is the inverse}, we know that $f$ has an $\ell^{2}$ formal inverse if and only if
\[\sum_{n,m\geq 0}\tau(x^{n+m})<\infty.\]
Since the above sum is easily seen to be
$\sum_{k=0}^{\infty}(k+1)\tau(x^{k}),$
the proof is complete.

\end{proof}




The combination of Lemma \ref{C:reductoin to positive case} and Lemma \ref{L:easier in sa case} reduces our problem to showing that if $x\in \R(G)$ with $\widehat{x}\in \Prob(G),$ then $\tau((x^{*}x)^{k})$ decays quickly. Recall that $\mu=\widehat{x}$ is a probability measure on $G,$ so $\nu=\frac{\mu^{*}+\mu}{2}$ is also a probability measure on $G,$ which is now symmetric. Given such a measure, one can form the random walk on $G$ which is a $G$-valued discrete time process $(X_{n})_{n=0}^{\infty}$ with $X_{0}=1$ and so that $X_{n+1}=X_{n}S_{n},$ where $(S_{n})_{n=1}^{\infty}$ are independent,  random elements of $G$ each with distribution $\nu.$ It is easily seen that $\tau((x^{*}x)^{k})$ is the probability that $X_{k}=1,$ i.e. that this random walk returns to the identity after $k$ steps. There are well known results, due to Varopolous, which give a precise relation between the decay rate of this probability and the growth of the group $G.$  Because of this, we easily obtain the following.

\begin{cor}\label{C:well balanced superpoly growth inverses}
Let $G$ be a countable, discrete, group and let $f\in \R(G)$ be semi-lopsided. Let $S=\supp(\widehat{f})\setminus\{1\}.$  Suppose that $(S\cup S^{-1})^{2}$ generates an infinite group. Let $H=\ip{S}.$
\begin{enumerate}[(i)]
\item If $f$ is lopsided, then $f$ has an $\ell^{1}$ formal inverse. \label{I:lopsided inverses easy case no growth}
\item If $f$ is not lopsided, and if $H$ either has super polynomial growth, or polynomial growth of degree at least $5,$ then $f$ has an $\ell^{2}$ formal inverse. \label{I:semi-lopsided inverses case with growth}
\end{enumerate}

\end{cor}

\begin{proof}
Let $m=\tau(f),$ and write $f=m(1-x).$ In each case, we examine the invertibility of $(1-x).$

(\ref{I:lopsided inverses easy case no growth}): This is well known, but we repeat the proof here. In this case, $\|\widehat{x}\|_{1}<1,$ and so by standard Banach algebra theory we know that $1-\widehat{x}$ has a $\ell^{1}$-convolution inverse $\xi.$ By definition, this means that $(1-x)\xi=\xi(1-x)=\delta_{1}.$

(\ref{I:semi-lopsided inverses case with growth}):
By  Lemma \ref{L:what is the inverse} and Lemma \ref{C:reductoin to positive case}, we may assume that $x=x^{*}$ and that $\widehat{x}\geq 0.$ By Lemma \ref{L:easier in sa case}, it suffices to show that
$\sum_{k}k\tau(x^{k})<\infty.$
Set $S=\supp(\widehat{x}).$ By assumption there is a constant $C>0$ with $|(S\cup \{1\})^{n}|\geq Cn^{5}$ for all $n\in \N.$ By \cite[Theorem 3]{VarRWGroups} this implies that there is a constant $A>0$ so that $\tau(x^{k})\leq Ak^{-5/2}$ for all $k\in \N.$ Hence
\[\sum_{k}k\tau(x^{k})\leq A\sum_{k}k^{-3/2}<\infty.\]

\end{proof}

We remark that if $x=x^{*}\in \R(G)$ with $\widehat{x}\in \Prob(G),$ and $\ip{\supp(\widehat{x})}$ has polynomial growth of degree $d,$ then there is a constant $B>0$ so that $\tau(x^{k})\geq Bk^{-d/2}$ for all $d\in \N$ (see \cite[Corollary 1.9]{RWAlex}). Thus, in the self-adjoint case, the assumption that $\ip{\supp(\widehat{x})}$ has either superpolynomial growth or polynomial growth of degree at least $5$ is optimal. Unfortunately, we do not know if the assumption that $\ip{\supp(\widehat{x})}$ has either superpolynomial growth or polynomial growth of degree at least $5$ is optimal in the case  that $p$ is not self-adjoint. That is to say, it is possible that there is a group $G,$ and a $x\in \R(G)$ with $\widehat{x}\in \Prob(G),$ so that $\ip{\supp(\widehat{x})}$ has polynomial growth of degree at most $4,$ and so that $1-x$ has an $\ell^{2}$ formal inverse.

At this stage, we have addressed why the growth assumption on $G$ is needed. In the next  subsection, we will address why the assumption of orderability of $G$ is important.

\subsection{Orderability and  Fourier Transforms}\label{S:orderability}

In the previous section, we saw that we could put mild assumptions on the growth of $G$ in order to guarantee that any semi-lopsided $f$ whose support generates $G$ has an $\ell^{2}$ inverse. We thus turn to addressing the second part of exhibiting $m_{X_{f}}$ as a factor of a Bernoulli measure: finding conditions on $G,\nu$ so that  the Fourier transform of $(\Theta_{\xi})_{*}(\nu^{\otimes G})$ is $	1_{\Z(G)f}.$

It may be helpful to sketch what the difficulty is here. Suppose $\nu\in \Prob(\Z)$ has mean zero and a finite second moment, and that $f$ has an $\ell^{2}$ formal inverse $\xi.$ Set $\mu_{\xi}=(\Theta_{\xi})_{*}(\nu^{\otimes G}).$ Let $\alpha\in\Z(G)$ with $\alpha\notin \Z(G)f,$ then
\[\widehat{\mu}_{\xi}(\alpha)=\prod_{g\in G}\widehat{\nu}((\alpha \xi)(g)).\]
Note that this is an absolutely converging product. So this product will be zero if and only if $\widehat{\nu}((\alpha\xi)(g_{0}))=0$ for some $g_{0}\in G.$ So we try to find such a $g_{0}$ (which will depend upon $\alpha$). As $\nu\in \Prob(\Z),$ we know that $\widehat{\nu}$ is identically $1$ on $\Z,$ so necessarily we must find a $g_{0}$ so that $(\alpha\xi)(g_{0})\notin \Z.$ Fortunately, the fact that $\xi$ is an $\ell^{2}$ formal inverse to $f$ and  that $\alpha\notin \Z(G)f$ is enough to guarantee that there is a $g_{0}$ with $(\alpha\xi)(g_{0})\notin \Z.$ So now we have forced $\widehat{\nu}((\alpha\xi)(g_{0}))$ to be less than $1$ in absolute value. Of course, this is not enough. We need to force $\widehat{\nu}((\alpha\xi)(g_{0}))$ to be zero. So we need to find probability measures on $\Z$ whose Fourier transforms vanish reasonably often. Given $m\in \N,$ it is not hard to exhibit a $\nu\in \Prob(\Z)$ with mean zero and finite  second moment which has $\widehat{\nu}=0$ on $(\frac{1}{m}\Z)\cap \Z^{c},$ e.g.
\[\nu=\begin{cases}
u_{\{-k,\cdots,k\}},& \textnormal{ if $m=2k+1$ is odd}\\
u_{\{-m,\cdots,-1\}}*u_{\{1,\cdots,m\}}& \textnormal{if $m$ is even}.
\end{cases}.\]
So we need to ensure that there is some $g_{0}\in G$ so that the $(\alpha\xi)(g_{0})$ is not an integer, and that its denominator is ``not too big." The following lemma helps us in this regard by allowing us to assume that the coefficients of $\alpha$ are not very big.

\begin{lem}\label{L:coefficents arent big}
Let $G$ be a countable, discrete group, and let $f\in \Z(G)$ be semi-lopsided. Suppose that $f$ has a $c_{0}$ formal right inverse. Set $m=\tau(f),$ and let $\alpha\in\Z(G).$
\begin{enumerate}[(i)]
\item If $m>\sum_{g\in G\setminus\{1\}}\widehat{f}(g),$ then we may write $\alpha=\beta+cf$ where $\beta,c\in \Z(G),$ and $\|\widehat{\beta}\|_{\infty}\leq m-1.$ \label{I:nicer balanced case}
\item If $m=\sum_{g\in G\setminus\{1\}}\widehat{f}(g),$ then we may write $\alpha=\beta+cf$ where $\beta,c\in \Z(G),$ and $\widehat{\beta}(g)\in \{-m,\cdots,m-1\}$ for every $g\in G.$ Moreover, we may choose $\beta,c$ so that if $\widehat{\beta}(g)=-m,$ then $\widehat{\beta}(gs^{-1})< 0$ for every $s\in \supp(\widehat{f})\setminus\{1\}.$ \label{I:annoying balanced case}
\end{enumerate}

\end{lem}

\begin{proof}
Let $S=\supp(\widehat{f})\setminus\{1\}.$
Let $\xi$ be a $c_{0}$ formal right inverse of $f,$ since $\alpha\xi\in c_{0}(G,\R)$ we may write $\alpha\xi=x+\widehat{c}$ where $x\in c_{0}(G)\cap [-1/2,1/2)^{G}$ and $c\in \Z(G).$ Right multiplying by $f$ we have that
\[\widehat{\alpha}=xf+\widehat{cf},\]
and this shows that $xf\in c_{0}(G,\R)\cap \Z^{G}=c_{c}(G,\Z).$ So we may write $xf=\widehat{\beta}$ for some $\beta\in\Z(G).$ We show that $\beta$ has the desired properties in each case.

(\ref{I:nicer balanced case}):
This case is divide into two subcases. First, suppose that $f$ is lopsided, so $\sum_{s\in S}|\widehat{f}(s)|<m.$ Then for every  $g\in G,$
\[|\widehat{\beta}(g)|=\left|mx(g)-m\sum_{s\in S}x(gs^{-1})\widehat{f}(s)\right|\leq \frac{1}{2}\left(m+\sum_{s\in S}|\widehat{f}(s)|\right)<m.\]
So $\widehat{\beta}(g)\in (-m,m)\cap \Z=\{-(m-1),\cdots,m-1\}.$ Thus $\|\widehat{\beta}\|_{\infty}\leq m-1.$

If $f$ is not lopsided, then $\sum_{s\in S}|\widehat{f}(s)|=m,$ and the fact that $\sum_{s\in S}\widehat{f}(s)<m$ implies that we can choose an $s_{0}\in S$ so that $\widehat{f}(s_{0})<0.$
Fix $g\in G.$ Then
\[\widehat{\beta}(g)=mx(g)+\sum_{s\in S}x(gs_{0}^{-1})\widehat{f}(s)<\frac{m}{2}+\frac{1}{2}\sum_{s\in S}|\widehat{f}(s)|=m.\]
Additionally,
\[\widehat{\beta}(g)=mx(g)+x(gs_{0}^{-1})\widehat{f}(s_{0})+\sum_{s\in S,s\ne s_{0}}x(gs^{-1})\widehat{f}(s).\]
The fact that $\widehat{f}(s_{0})<0$ and $x(gs_{0}^{-1})\in [-1/2,1/2)$ implies that $x(gs^{-1})\widehat{f}(s_{0})>-\frac{1}{2}|\widehat{f}(s_{0})|.$ Thus
\[\widehat{\beta}(g)>-\frac{m}{2}-\frac{1}{2}|\widehat{f}(s_{0})|-\frac{1}{2}\sum_{s\in S,s\ne s_{0}}|\widehat{f}(s)|=-m.\]
So $-m<\widehat{\beta}(g)<m$ and as in the first subcase the fact that $\widehat{\beta}$ is integer valued implies that $\|\widehat{\beta}\|_{\infty}\leq m-1.$

(\ref{I:annoying balanced case}):
In this case, we must have that $\widehat{f}(s)>0$ for every $s\in S.$ Fix $g\in G.$ Since $\widehat{f}(s)>0$ for every $s\in S,$
\[\widehat{\beta}(g)=mx(g)+\sum_{s\in S}x(gs^{-1})\widehat{f}(s)<\frac{1}{2}\left(m+\sum_{s\in S}\widehat{f}(s)\right)=m.\]
Similarly,
\[\widehat{\beta}(g)\geq \frac{-m}{2}-\frac{1}{2}\sum_{s\in S}\widehat{f}(s)=-m.\]
So $\widehat{\beta}(g)\in [-m,m)\cap \Z=\{-m,\cdots,m-1\}.$ It simply remains to show that if $\widehat{\beta}(g)=-m,$ then $\widehat{\beta}(gs^{-1})<0$ for every $s\in S.$

So suppose that $\widehat{\beta}(g)=-m.$ Then
\[\widehat{\beta}(g)=mx(g)+\sum_{s\in S}x(gs^{-1})\widehat{f}(s).\]
Since $\widehat{f}(s)\geq 0$ for all $s\in S,$  $\sum_{s\in S}\widehat{f}(s)=m,$ and $x\in [-1/2,1/2)^{G},$ the fact that $\widehat{\beta}(g)=-m$ forces $x(gs^{-1})=-1/2$ for every $s\in S.$ Now fix $s\in S.$ Using that $\widehat{f}(t)>0$ for all $t\in S,$ and that $x\in [-1/2,1/2)^{G},$ we have:
\[\widehat{\beta}(gs^{-1})=-\frac{m}{2}+\sum_{t\in S}x(gs^{-1}t^{-1})\widehat{f}(t)<-\frac{m}{2}+\frac{1}{2}\sum_{t\in S}\widehat{f}(t)=0.\]
So $\widehat{\beta}(gs^{-1})<0.$

\end{proof}

We now explicitly discuss where orderability comes into play. We will work with something slightly more general than a left-invariant total order on $G.$

\begin{defn}
Let $G$ be a countable, discrete group. We say that a partial order $\preceq$ on $G$ is \emph{left-invariant} if whenever  $h_{1},h_{2}\in G$ with $h_{1}\preceq h_{2},$ then for all $g\in G$ we have that $gh_{1}\preceq gh_{2}.$
If there is a left-invariant total order on $G,$ then we say that $G$ is \emph{left-orderable}.
\end{defn}

For example, it is easy to exhibit a left-invariant  partial order on $\F_{2}=\ip{a,b}.$ Let $P$ be the set of elements of $\F_{2}$ which are products of $a,b$ (no $a^{-1},b^{-1}$ occur in its word decomposition), with the convention that $1\notin P.$ We can then define a partial order $\preceq$ by demanding that $h_{1}\preceq h_{2}$ if and only if $h_{1}^{-1}h_{2}\in P\cup\{1\}.$ This is a partial order on $\F_{2},$ and it makes the generators \emph{order positive} (i.e. larger than $1$ in this partial order). It is a fact that there is a left-invariant total order on $\F_{2}$ which make the generators order positive, but it is harder to construct.

\begin{lem}\label{L:getting divisibility}
Let $G$ be a countable, discrete group, and let $f\in \Z(G)$ be semi-lopsided, and set $m=\widehat{f}(1).$ Suppose that $f$ has an $\ell^{2}$ formal inverse $\xi$. Let $S=\supp(\widehat{f})\setminus\{1\},$ and assume that $H=\ip{S}$ has a left-invariant partial order so that $S\subseteq \{h\in H:h\succ 1\},$ and that $\ip{h_{1}^{-1}h_{2}:h_{1},h_{2}\in S}$ is infinite. If $\alpha\in \Z(G),$ but $\alpha\notin \Z(G)f,$ then there is a $g\in G$ so that $(\alpha\xi)(g)\in\frac{1}{m}\Z\cap \Z^{c}.$

\end{lem}

\begin{proof}
Let $P=\{h\in H:h\succ 1\},$ since $\preceq$ is a partial order on $H$ we have that $P\cap P^{-1}=\varnothing.$ Since $P\cap P^{-1}=\varnothing,$ we may then extend the partial order on $H$ to a partial order on $G$ by saying that $g\preceq h$ if $g^{-1}h\in P\cup\{1\}.$ It is easy to see that this partial order on $G$ is left-invariant so we may assume, without loss of generality, that $G$ has a left-invariant partial order which makes the elements of $S$ positive.

Note that if $\alpha-\beta\in \Z(G)f,$ and if $g\in G$ satisfies $(\beta\xi)(g)\in \frac{1}{m}\Z\cap \Z^{c},$ then $(\alpha\xi)(g)\in\frac{1}{m}\Z\cap \Z^{c}.$
So by Lemma \ref{L:coefficents arent big} we may, and will, assume that one of the following two cases hold. Either
\begin{enumerate}[(a)]
    \item $\|\widehat{\alpha}\|_{\infty}\leq m-1,$ or \label{I:control of coeff good cases}
    \item $\widehat{\alpha}(g)\in \{-m,\cdots,m-1\}$ for every $g\in G.$ Further, if $\widehat{\alpha}(g)=-m,$ then $\widehat{\alpha}(gs^{-1})<0$ for every $s\in S.$ \label{I:control of coeff bad cases}
\end{enumerate}
Let $g$ be an element of $\supp(\alpha)$ which is minimal with respect to this partial order. We make the following claim.

\emph{Claim: $|\widehat{\alpha}(g)|\leq m-1.$}

To prove the claim, we first note that if (\ref{I:control of coeff good cases}) holds then the claim is trivial, so we may assume that (\ref{I:control of coeff bad cases}) holds. Note that $(gs^{-1})^{-1}g=s\in P$ for every $s\in S.$ So $gs^{-1}\prec g$ and $gs^{-1}\ne g$ for every $s\in S.$ So by minimality of $g$ we must that $\widehat{\alpha}(gs^{-1})=0$ for every $s\in S.$ So (\ref{I:control of coeff bad cases}) now implies that $|\widehat{\alpha}(g)|\leq m-1.$

We now return to the proof of the lemma. Let $x=-\frac{1}{m}\sum_{s\in S}\widehat{f}(s)s,$ so $f=m(1-x).$ Then, by Lemma \ref{L:what is the inverse} we have that
\[(\alpha\xi)(g)=\frac{\widehat{\alpha}(g)}{m}+\frac{1}{m}\sum_{n=1}^{\infty}(\alpha x^{n})(g).\]
Fix $n\geq 1,$ then
\[(\alpha x^{n})(g)=\sum_{s\in S^{n}}\alpha(gs^{-1})x^{n}(s).\]
Let $s\in S^{n}.$ Then $s\in P$ and $s\ne 1$ since $S\subseteq \{h\in H:h\succ 1\}.$ Thus $(gs^{-1})^{-1}g=s\in P.$ So $gs^{-1}\preceq g$ and $gs^{-1}\ne g,$ since $s\ne 1.$ By minimality we thus have that $\alpha(gs^{-1})=0$ for all $s\in S^{n},$ and thus $(\alpha x^{n})(g)=0$ for all $n\geq 1.$ So by the claim,
\[(\alpha\xi)(g)=\frac{\widehat{\alpha}(g)}{m}\in \left\{-\frac{(m-1)}{m},-\frac{(m-2)}{m},\cdots,\frac{m-1}{m}\right\}.\]
Since $g\in \supp(\widehat{\alpha}),$ it follows that
\[(\alpha\xi)(g)\in \left\{-\frac{(m-1)}{m},-\frac{(m-2)}{m},\cdots,\frac{m-1}{m}\right\}\setminus \{0\}\subseteq \frac{1}{m}\Z\cap \Z^{c}.\]

\end{proof}

\begin{cor}\label{C:factor of a Bernoulli shift main one}
Let $G$ be a countable, discrete group, and let $f\in \Z(G)$ be semi-lopsided. Suppose that $f$ has an $\ell^{2}$ formal inverse. Suppose that $S=\supp(\widehat{f})\setminus\{1\},$ and that $H=\ip{S}$ has a left-invariant partial order $\preceq$ so that $S\subseteq \{h\in H:h\succ 1\},$ and that $\ip{h_{1}^{-1}h_{2}:h_{1},h_{2}\in S}$ is infinite. Then $G\actson (X_{f},m_{X_{f}})$ is a factor of a Bernoulli shift.
\end{cor}

\begin{proof}
Let $\xi$ be the $\ell^{2}$ formal inverse to $f,$ and $m=\tau(f).$
First assume that $m$ is an odd integer, and write $m=2k+1.$ We can let $\Theta_{\xi}\colon \{-k,\cdots,k\}^{G}\to \T^{G}$ be defined as in Theorem \ref{T: background convolution}  corresponding to $\nu=u_{\{-k,\cdots,k\}}.$ Let $\mu=(\Theta_{\xi})_{*}(\nu^{\otimes G}).$ So for every $\alpha\in \Z(G):$
\[\widehat{\mu}(\alpha)=\prod_{g\in G}\widehat{\nu}((\alpha\xi)(g))=\prod_{g\in G}\frac{\sin((2k+1)\pi(\alpha\xi)(g))}{(2k+1)\sin(\pi(\alpha\xi)(g))}.\]
Suppose $\alpha\in \Z(G)f,$ and write $\alpha=\beta f.$ Then $\alpha\xi=\widehat{\beta}\in c_{c}(G,\Z).$ Hence we have $(\alpha\xi)(g)\in \Z$ for all $g\in G.$ Since $\widehat{\nu}(l)=1$ for every $l\in \Z,$ we have that $\widehat{\mu}(\alpha)=1.$ Suppose that $\alpha\in \Z(G),$ but $\alpha\notin \Z(G)f.$ Then by Lemma \ref{L:getting divisibility}, there is some $g\in G$ so that $(\alpha\xi)(g)\in \frac{1}{m}\Z \cap \Z^{c}.$ So $\widehat{\mu}((\alpha\xi)(g))=0$ (since $\widehat{\nu}(t)=\frac{\sin((2k+1)\pi t)}{(2k+1)\sin(\pi t))}$), and thus $\widehat{\mu}(\alpha)=0.$ Hence we know that $\widehat{\mu}=1_{\Z(G)f},$ and this is equivalent to saying that $\mu=m_{X_{f}}.$

Now assume that $m$ is even. Let $\eta=u_{\{0,\cdots,m-1\}},$ and set $\nu=\eta^{*}*\eta.$ Observe that $\mu$ is a measure on $\{-(m-1),\cdots,m-1\},$ and since $\widehat{\nu}=|\widehat{\eta}|^{2}$ we have that
\[\widehat{\nu}(t)=\begin{cases}
1,& \textnormal{ if $t\in \Z$}\\
\left|\frac{\sin(m\pi t)}{m\sin(\pi t)}\right|^{2},& \textnormal{ if $t\notin \Z.$}
\end{cases}\]
Moreover, it is direct to check that $\nu$ has mean zero. Let $\Theta_{\xi}\colon \{-(m-1),\cdots,m-1\}^{G}\to \T^{G}$ be defined as in Theorem \ref{T: background convolution}  for this $\nu$, and set $\mu=(\Theta_{\xi})_{*}(\mu^{\otimes G}).$ Then
\[\widehat{\mu}(\alpha)=\prod_{g\in G}\widehat{\nu}((\alpha\xi)(g))\]
for all $\alpha\in \Z(G).$
First suppose that $\alpha\in \Z(G)f.$ Then as in the case that $m$ is odd, we know that $\alpha\xi\in c_{c}(G,\Z),$ and thus $\widehat{\mu}(\alpha)=1.$ Now assume that $\alpha\in \Z(G),$ but that $\alpha\notin \Z(G)f.$ As in the  case that $m$ is odd, we may find a $g\in G$ so that $(\alpha\xi)(g)\in \frac{1}{m}\Z\cap \Z^{c}.$ For such a $g$ we have that $\widehat{\nu}((\alpha\xi)(g))=0,$ and thus $\widehat{\mu}(\alpha)=0.$ So $\widehat{\mu}=1_{\Z(G)f},$ and thus $\nu=m_{X_{f}}.$

\end{proof}

\begin{cor}\label{C:factor of Bern special case}
Let $G$ be a countable, discrete group, and  $f\in \Z(G)$ be semi-lopsided. Suppose that there is a left-invariant partial order $\preceq$ on the group so that $\supp(\widehat{f})\setminus \{1\}\subseteq \{g\in G:g\succ 1\}.$ Assume that $\ip{a^{-1}b:a,b\in \supp(\widehat{f})\setminus\{1\}}$ is infinite.  Set $H=\ip{\supp(\widehat{f})\setminus\{1\}}.$ Suppose that one of the following three cases hold:
\begin{enumerate}[(i)]
\item either $f$ is lopsided, or
\item  $H$ has super-polynomial growth, or
\item $H$ has polynomial growth of degree at least $5.$
\end{enumerate}
Then $G\actson (X_{f},m_{X_{f}})$ is a factor of a Bernoulli shift.

\end{cor}

As we will see in the next section, our assumptions imply that the kernel $N$ of $G\actson X_{f}$ is finite, and that $G/N\actson (X_{f},m_{X_{f}})$ is essentially free. Thus, by \cite{OrnWeiss}, if $G$ is amenable,  then we can say that $G/N\actson (X_{f},m_{X_{f}})$ is a Bernoulli shift.

\section{Sample Applications}\label{S:applications}

In this section, we shall give many examples of groups $G$ and semi-lopsided elements $f\in\Z(G)$ so that $G\actson (X_{f},m_{X_{f}})$ is a factor of a Bernoulli shift.  We have divided this section into subsections: one where the examples have the acting group amenable, and one where it is not (the second section should maybe be titled ``potentially nonamenable groups" as it includes Thompson's group and the question of whether or not it is amenable is open). This is because in the amenable case we can apply Ornstein theory to show that this actions are in fact Bernoulli (once we argue that these actions are essentially free). While we cannot do this in the nonamenable case (there is a counterexample due to Popa in \cite[Corollary 2]{PopaCohomologyOE}, see also \cite{AustinGeo}), the fact that these actions are factors of Bernoulli shifts still has lots of interesting consequences such as showing that they have completely positive entropy, are solidly ergodic,  are mixing of all orders, and have countable Lebesgue spectrum.

Most of examples will have the acting group be left-orderable. In this case, the assumption that $\ip{a^{-1}b:a,b\in \supp(\widehat{f})\setminus\{1\}}$ is infinite is always satisfied except in the case that $\supp(\widehat{f})$ has size $2.$  In both the amenable and nonamenable case we will also give examples of groups which are not torsion-free for which we can explicitly write down left-invariant partial orders, and this  allow us to give nice examples of principal algebraic actions  which are factors of Bernoulli shifts. Typically in this situation, the assumption that $\ip{a^{-1}b:a,b\in \supp(\widehat{f})\setminus\{1\}}$ is infinite is also straightforward, and we will not explicitly prove it except in the cases where it is slightly less obvious.

For amenable $G,$ demanding that $G$ be left-orderable and finitely generated implies that $G$ is locally indicable (i.e. every finitely-generated subgroup of $H$ has a surjective homomorphism onto $\Z$), by \cite{MorrisLeftOrd}. So this rules out some possibilities for $G,$ for example it cannot be simple.  By results of \cite{HydeLodLeftOrd}, there exists a continuum size collection of pairwise nonisomorphic, simple, nonamenable, finitely generated, left-orderable groups (this follows up previous related work in \cite[Theorem 1.7]{KimKobLodLeftOrd}), and so nonamenable left-orderable groups can be simple. We do not know if there are examples of simple, finitely generated, amenable, groups $G$ which have a positive semigroup $P$ so that $\ip{P}=G$. We will give an example later of a group $G,$ so that $\ip{P}\ne G$ for every positive semigroup $P\subseteq G,$ but so that it has a left-orderable subgroup of finite-index. So it is not always the case that a group can be generated by a positive semigroup, even if it has a ``large" left-orderable subgroup.

In each case the elements $f\in\Z(G)$ we construct will have the property that $\supp(\widehat{f})$ generates $G,$ even though this is not necessary to apply Corollary \ref{C:factor of Bern special case}. There are three primary reasons for this. The first is that, in the semi-lopsided case, it makes it more transparent that $\ip{\supp(\widehat{f})}$ has fast enough growth. Of course this is not an issue if we stick to lopsided elements, but this removes the harmonic model examples which are quite interesting for their connection with random walks. The second two reasons are as follows: suppose we take $f\in\Z(G),$ and let $H$ be the group generated by the support of $\widehat{f}.$ For clarity, we let $X_{f,H}$ be the Pontryagin dual of $\Z(H)/\Z(H)f,$ and $X_{f,G}$ be the Pontryagin dual of $\Z(G)/\Z(G)f.$ Then $G\actson (X_{f,G},m_{X_{f,G}})$ is the \emph{coinduced action} of $H\actson (X_{f,H},m_{X_{f,H}})$ (see \cite[Section 6]{Me5} for the terminology and a proof of this). The coinduction of a Bernoulli shift is a Bernoulli shift, and a factor map between $H$-actions functorially induces a factor map between the corresponding $G$-actions. Thus if $H\actson (X_{f,H},m_{X_{f,H}})$ is a Bernoulli shift (respectively a factor of a Bernoulli shift), then $(G\actson (X_{f,G},m_{X_{f,G}})$ will be a Bernoulli shift (respectively a factor of a Bernoulli shift). This gives us two more reasons to restrict to the case $\supp(\widehat{f})=G.$ The first is that we can do so without loss of generality, once we know the case when the support of $\widehat{f}$ generates $G,$ the case when it does not follows by a simple application of the coinduction construction. The second is that while we can create examples where $G\actson (X_{f},m_{X_{f}})$ is a Bernoulli shift and $\ip{\supp(\widehat{f})}$ is not all of $G,$ many of these will be factors of Bernoulli (or factors of Bernoulli) for not very interesting reasons. For example, we can view $\Z$ inside of $\F_{2}$ (in any number of ways). Certainly we know of many  $f\in \Z(\Z)$ for which $G\actson (X_{f,\Z},m_{X_{f,\Z}})$ is Bernoulli, (e.g. by \cite{KatzErgAut, LindErgAut, RudolphSchmidtCPE} this is true as long as it is ergodic) and so by coinduction we know that $\F_{2}\actson (X_{f,\F_{2}},m_{X_{f,\F_{2}}})$ is Bernoulli. Since the group structure of $\F_{2}$ does not enter in a nontrivial way in the proof that $\F_{2}\actson (X_{f,\F_{2}},m_{X_{f,\F_{2}}})$ is Bernoulli, examples constructed in this manner are not very satisfactory. Similar remarks apply to any other group with infinite amenable subgroups. Of course, once one exhibits a set of generators for the group, then it is easy to construct several other sets of generators. In most of the examples we give we will typically only consider $f\in\Z(G)$ so that $\supp(\widehat{f})=\{1\}\cup S$ where $S$ are some ``well known" generators of the group. We will leave it to the reader to modify these sets of generators and create many more examples of Bernoulli (or factor of Bernoulli) principal algebraic actions. There are certainly an endless number of ways of doing this.

We remark that once we demand that $\ip{\supp(\widehat{f})}=G$ the assumption that there is a left-invariant partial order $\preceq$ on $G$ so that $\supp(\widehat{f})$ is contained in the corresponding positive semigroup becomes an actual restriction on the group. Of course, such a group cannot be torsion. But even if $G$ has a finite-index, left-orderable subgroup, it is not necessarily the case that $G$ has a positive semigroup $P$ with $\ip{P}=G.$  For example, consider the unique nontrivial action $\Z/2\Z\actson \Z$ by automorphisms, and set $G=\Z\rtimes \Z/2\Z.$ There is no positive semigroup $P\subseteq G$ with the property that $\ip{P}=G.$ This is because any positive semigroup $P\subseteq G$ must have no torsion elements, and this forces $P\subseteq \Z.$ So our methods do not apply to any principal algebraic action of this group.

\subsection{The Amenable Case}

Corollary \ref{C:factor of Bern special case} is most striking when $G$ is amenable, since in this case being a factor of a Bernoulli shift implies that the action \emph{is} a Bernoulli shift (at least when one quotients by the kernel of the action). This is a consequence of Ornstein theory. However, as Ornstein theory only applies to free actions of groups we should first observe that the actions we are considering are free after modding out by the kernel.

\begin{prop}\label{P:free for free}
Let $G$ be a countable, discrete group and suppose that $f\in \Z(G)$ has a $c_{0}$ formal inverse. If $N$ is the kernel of the action $G\actson X_{f},$ then $N$ is finite and $G/N\actson (X_{f},m_{X_{f}})$ is essentially free.

\end{prop}

\begin{proof}
As was pointed out in \cite{BowenLi}, the fact that $f$ has a $c_{0}$ formal inverse implies that $G\actson (X_{f},m_{X_{f}})$ is mixing. This makes it obvious that $N$ is finite. Additionally, the fact that $G\actson (X_{f},m_{X_{f}})$ is mixing forces $G/N\actson (X_{f},m_{X_{f}})$ to be essentially free by \cite{RobinAF}.

\end{proof}

So in the amenable case, we can always conclude that if $G\actson (X_{f},m_{X_{f}})$ is a factor of a Bernoulli shift.
In most of our examples, $G$ is additionally torsion-free. In this case, we have that $N$ is trivial, i.e. the algebraic action is faithful. If $G$ is assumed torsion-free, there is  a more elementary proof that $G\actson (X_{f},m_{X_{f}})$ is essentially free by appealing to \cite[Appendix A]{MeMaxMinWC}. This does not use the Feit-Thompson theorem as in \cite{RobinAF}. By Proposition \ref{P:free for free} and the extension of Ornstein theory to amenable groups (see \cite{OrnWeiss}), if $G$ is amenable and $G\actson (X_{f},m_{X_{f}})$ is a factor of a Bernoulli shift, and if $N$ is the kernel of the action $G\actson X_{f},$ then $G/N\actson (X_{f},m_{X_{f}})$ is isomorphic to a Bernoulli shift (this follows from \cite[III.4 Proposition 1, III.6 Theorem 2, III.5 Corollary 5]{OrnWeiss}). We will give several examples of Bernoulli principal algebraic actions of amenable groups shortly, but let us first point out explicitly here that this annoyance with finite normal subgroups can occur.

\begin{prop}
 Let $G$ be a countable, discrete, group and let $N\triangleleft G$ be finite with $N\ne\{1\}.$ Let $b$ be an integer, and let
 \[f=1+b\sum_{n\in N}n.\]
 Then $f$ has an inverse in $\Q(G),$ and $N$ is the kernel of $G\actson X_{f}.$ Further, the induced action $G/N\actson (X_{f},m_{X_{f}})$ is isomorphic to the Bernoulli action $G/N\actson (\{1,\cdots,1+b|N|\},u_{1+b|N|})^{G/N}.$
\end{prop}

\begin{proof}
Let $x=\sum_{n\in N}n,$ $e=\frac{1}{|N|}x,$ so $e^{2}=e$ and we can write $f$ as
\[f=(1-e)+(1+b|N|)e.\]
Simple calculations show that $f$ has an inverse in $\Q(G)$ given by
\[\phi=(1-e)+(1+b|N|)^{-1}e.\]
Moreover, it is easy to check that the normality of $N$ implies that $f,$ and thus $\phi,$ is central in $\Q(G).$

To check that $N$ acts trivially on $X_{f},$ let $n\in \N.$ It then suffices to show that $(n-1)\Z(G)\subseteq \Z(G)f.$ So let $\alpha\in \Z(G).$ Then by direct calculation,
\[(n-1)\alpha\phi=(n-1)\phi\alpha=(n-1)\alpha\in\Z(G).\]
So $(n-1)\alpha=(n-1)\alpha\phi f\in \Z(G)f.$ So $N$ acts trivially on $X_{f}.$

It simply remains to prove that $G/N\actson (X_{f},m_{X_{f}})$ is Bernoulli. Let $J$ be the left ideal in $\Z(G)$ generated by $\{n-1:n\in \N\},$ by normality of $N$ this is in fact a two sided ideal. It is easy to see that we have an isomorphism of $\Z(G)$ modules $\Psi\colon \Z(G)/J\to \Z(G/N)$ given by $(\Psi(\alpha+J))^{\widehat{}}(gN)=(\alpha x)^{\widehat{}}(g).$ Since we already saw that $\Z(G)f\supseteq J,$ we get a natural $\Z(G)$-modular surjection $\pi\colon \Z(G)/J\to \Z(G)/\Z(G)f,$ and so we have a natural $\Z(G)$-modular surjection $\Z(G/N)\to \Z(G)/\Z(G)f$ given by $q=\pi\circ \Psi^{-1}.$

It remains to compute the kernel of $q.$ Note that if $\alpha\in \Z(G),$ then $\alpha\in \Z(G)f$ if and only if $\alpha\phi\in \Z(G).$ By a direct computation, $\phi=1+\frac{b}{1+b|N|}x.$ So
\[\alpha\phi=\alpha+\frac{b}{1+b|N|}\alpha x\]
and thus $\alpha\phi\in\Z(G)$ if and only if $\frac{b}{1+b|N|}\alpha x\in\Z(G).$ By definition, this means that $\frac{b}{1+b|N}\Psi(\alpha)\in \Z(G/N).$ Since $b,1+b|N|$ are easily seen to be relatively prime, this is the same as saying that $\Psi(\alpha)\in(1+b|N|)\Z(G/N).$ So we have shown that $\ker(q)=(1+b|N|)\Z(G/N).$

So we have exhibited an isomorphism $\Z(G)/Z(G)f\cong \Z(G/N)/(1+b|N|)\Z(G/N)$ of $\Z(G)$-modules, and applying Pontryagin duality shows that $G\actson (X_{f},m_{X_{f}})\cong G\actson (\Z/(1+b|N|)\Z,u_{\Z/(1+b|N|)\Z})^{G/N}.$ So $G/N\actson (X_{f},m_{X_{f}})$ is isomorphic to the Bernoulli shift $G/N\actson (\{1,\cdots,1+b|N|\},u_{1+b|N|})^{G/N}.$

\end{proof}

Let us proceed to give several examples of semi-lopsided elements whose corresponding actions are Bernoulli. In many examples our groups will be torsion-free, and so it is a consequence of Proposition \ref{P:free for free} that the actions we are considering are essentially free. We will thus not explicitly reference that these actions are essentially free before applying Ornstein theory. In later examples, we will consider groups with torsion and will give an explicit argument that these actions are faithful (and thus free by Proposition \ref{P:free for free}).

Since we know in these examples that the actions are isomorphic to Bernoulli, it is nice to know what Bernoulli shift they are. Since Bernoulli shifts over amenable groups are completely classified by their entropy by \cite{OrnClassify1, OrnClassify2, OrnWeiss} (see \cite{Bow,BowenOrn, SewardOrn, Stepin} for the more general fact that Bernoulli shifts overs sofic groups are completely classified their entropy), once we compute the entropy of these actions we will know what Bernoulli shift they are isomorphic to. By \cite[Proposition 2.2]{Me5} we know that once $f$ has an $\ell^{2}$ formal inverse, then $f$ is injective as a convolution operator $\ell^{2}(G)\to \ell^{2}(G).$ So by \cite{LiThom} (see also \cite{Me5} for the sofic case), the entropy of $G\actson (X_{f},m_{X_{f}})$ is $\log \Det_{L(G)}(f).$ In  Appendix \ref{S:FKD} (see Corollary \ref{C:FKD Calculation}) it is shown that if $f\in\Z(G)$ is semi-lopsided, and if $\supp(\widehat{f})=\{1\}\cup S$ where $S\subseteq P$ for some positive semigroup $P\subseteq G,$ then $\log \Det_{L(G)}(f)=\log\tau(f).$ So in all of the examples we are considering, we know that the entropy of $G\actson (X_{f},m_{X_{f}})$ is $\log\tau(f).$

\begin{example}
Let $G$ be the Heisenberg group of upper triangular matrices with integer entries and all diagonal entries $1.$ Then $G$ has polynomial growth of order 4 by the Bass-Guivarc’h formula. So our results only allow us to use lopsided elements. The group $G$ has a presentation $G=\ip{a,b,c:[c,a]=1,[c,b]=1,[a,b]=c}.$ Every element of $G$ can be uniquely represented as $a^{n}b^{m}c^{k}$ for integers $n,m,k\in \Z.$ The group $G$ has a total order $<$ given by saying that  $a^{n_{1}}b^{m_{1}}c^{k_{1}}<a^{n_{2}}b^{m_{2}}c^{k_{2}}$ if either:
\begin{itemize}
\item $n_{1}<n_{2},$ or
\item $n_{1}=n_{2},$ and $m_{1}<m_{2},$
\item $n_{1}=n_{2},m_{1}=m_{2}$ and $k_{1}<k_{2}.$
\end{itemize}
It can be checked that this is a total order which is left and right invariant.
By Corollary \ref{C:factor of Bern special case} and Ornstein theory, if
\begin{itemize}
\item $f=k+na+mb$ with $n,m,k\in \Z$ and $|n|+|m|<k,$ or
\item $f=k+na+mb+lc$ with $n,m,l,k\in \Z$ and $|n|+|m|+|l|<k,$
\end{itemize}
then $G\actons (X_{f},m_{X_{f}})$ is isomorphic to $G\actson (\{1,\cdots,k\},u_{k})^{G}.$ Since we are only considering lopsided elements, if we set $P=\{g\in G:g>1\},$ then we can in fact take any pairwise distinct $x_{1},\cdots,x_{k}\in P,$ any $m,n_{1},\cdots,n_{k}\in \Z$ with $\sum_{j}|n_{j}|<m,$ and then
\[f=m+\sum_{j=1}^{k}n_{j}x_{j}\]
will be such that $G\actson (X_{f},m_{X_{f}})$ is a Bernoulli shift with entropy $\log(m),$ provided $k\geq 2$. The reason for the restriction $k\geq 2,$ is that once $k\geq 2$ we have that $\ip{\{x_{i}^{-1}x_{j}:1\leq i,j\leq k\}}$ is a nontrivial subgroup of $G,$ and as such this subgroup is infinite because $G$ is torsion-free.
\end{example}

\begin{example}\label{E:Higher d heisenberg}
We can generalize the previous example slightly. For $N\geq 1,$ let $H_{N}$ be the group of upper triangular $(N+2)\times (N+2)$-matrices with integer entries, $1's$ on the diagonal, and so that all other nonzero entries are on the first row or the last column. For $1\leq j,k\leq n,$ let $E_{j,k}$ be the matrix defined by $(E_{j,k})_{rs}=\delta_{r=j}\delta_{k=s}.$ Define $c\in G$ by $c=\id+E_{1,N+2},$ and for $1\leq i\leq N,$ define $a_{i},b_{i}$ by $a_{i}=\id+E_{1,i+1},$ $b_{i}=\id+E_{i+1,N+2}.$ Then \begin{itemize}
    \item $\ip{a_{1},\cdots,a_{N},b_{1},\cdots,b_{N},c}=G,$
    \item $\ip{a_{1},\cdots,a_{N}},\ip{b_{1},\cdots,b_{N}}$ are abelian,
    \item  $c$ is central in $G,$
    \item $[a_{i},b_{j}]=1$ if $i\ne j,$
    \item $[a_{i},b_{i}]=c.$
\end{itemize}
Every element of $g\in G$ can be uniquely represented as
\[a_{1}^{n_{1}}b_{1}^{n_{2}}a_{2}^{n_{3}}b_{2}^{n_{4}}\cdots a_{N}^{n_{2N-1}}b_{n}^{n_{2N}}c^{n_{2N+1}}\]
for integers $n_{1},\cdots,n_{N},$ $m_{1},\cdots,m_{N},$ $l\in \Z.$ We can order $G$ lexicographically as before:
\[a_{1}^{n_{1}}b_{1}^{n_{2}}a_{2}^{n_{3}}b_{2}^{n_{4}}\cdots a_{N}^{n_{2N-1}}b_{N}^{n_{2N}}c^{n_{2N+1}}<a_{1}^{m_{1}}b_{1}^{m_{2}}a_{2}^{m_{3}}b_{2}^{m_{4}}\cdots a_{N}^{m_{2N-1}}b_{N}^{m_{2N}}c^{m_{2N+1}}\]
if $n_{j}<m_{j}$ when $1\leq j\leq 2N+1$ is  the minimal index such that $n_{j}\ne m_{j}.$ It can again be checked that this is both left and right-invariant. If we set $P=\{g\in G:g>1\},$ then as before we can take $x_{1},\cdots,x_{k}\in P$ $m,m_{1},\cdots,m_{k}\in \Z$ with $\sum_{j}|m_{j}|<m$ and $f=m+\sum_{j}m_{j}x_{j}$ will have $G\actson (X_{f},m_{X_{f}})$ being a Bernoulli shift with entropy $\log(m).$

The growth of $H_{N}$ is $2(N+1)$ by the Bass-Guivarc’h formula, and so once $N\geq 2$ we even have semi-lopsided, but not lopsided examples. For example, with
\[f=(2N+1)-\left(\sum_{j=1}^{N}a_{j}+\sum_{j=1}^{N}b_{j}+c\right)\]
we have that $G\actson (X_{f},m_{X_{f}})$ is isomorphic to a Bernoulli shift with base entropy $\log(2N+1)$ (proved $N\geq 2$). As we show in Appendix \ref{S:inverses}, since $f$ is well-balanced, we know that $f$ does \emph{not} have an $\ell^{1}$ inverse. We show there as well, using that $G$ is amenable, that $\lambda(f)$ is not invertible.

We can of course add signs here and consider, e.g.,
\[f=(2N+1)-\left(\sum_{j=1}^{N}\varepsilon_{j}a_{j}+\sum_{j=1}^{N}\sigma_{j}b_{j}+\alpha c\right)\]
for any $\varepsilon_{1},\cdots,\varepsilon_{N},$ $\sigma_{1},\cdots,\sigma_{N},$ $\alpha$ in $\{\pm 1\}.$ We then still have that $G\actson (X_{f},m_{X_{f}})$ is a Bernoulli shift with entropy $\log(2N+1).$ Since $c\in \ip{a_{1},\cdots,a_{N},b_{1},\cdots,b_{N}},$ we can similarly consider
\[f=m+\left(\sum_{j=1}^{N}m_{j}a_{j}+\sum_{j=1}^{N}l_{j}b_{j}\right)\]
where $\sum_{j=1}^{N}|m_{j}|+\sum_{j=1}^{N}|l_{j}|\leq m,$ and $m_{1},\cdots,m_{N}$ $l_{1},\cdots, l_{N}$ are not zero. Then $G\actson (X_{f},m_{X_{f}})$ is a Bernoulli shift with entropy $\log(m).$

Similar examples, can be given by taking products of the generators. E.g. if $N=2,$ we may consider
\[f=4-(a_{1}+a_{2}+a_{1}b_{1}+a_{2}b_{2})\]

\end{example}

\begin{example}
For $N\geq 1,$ let $T_{N}$ be the group of upper-triangular $(N+2)\times (N+2)$-matrices with $1'$s on the diagonal. As $T_{N}$ contains $H_{N}$ from example \ref{E:Higher d heisenberg}, we know that $T_{N}$ has polynomial growth of degree at least $5$ once $N\geq 2.$ Define $c\in T_{N}$ by $c=\id+E_{1,N+2}$ and for $1\leq i<j\leq N+2$ with $(i,j)\ne (1,N+2),$ let $a_{ij}=\id+E_{i,j}.$ Then $T_{N}=\ip{a_{ij},c}$ and we leave it is an exercise to the reader argue as in Example \ref{E:Higher d heisenberg} to show that $T_{N}$ has an left-invariant total order which makes $a_{ij},c>1.$ Thus if
\begin{itemize}
    \item $(m_{ij})_{1\leq i<j\leq N+2,(i,j)\ne (1,N+2)}\in \Z\setminus\{0\},$
    \item $n\in \Z\setminus\{0\},$
    \item $m\in \N$ satisfies $|n|+\sum_{i,j}|m_{ij}|\leq m,$
\end{itemize}
then with $f=m+nc+\sum_{i,j}m_{ij}a_{ij}$ we have that $G\actson (X_{f},m_{X_{f}})$ is isomorphic to a Bernoulli shift with entropy $\log(m).$ If each $m_{ij},n>0$ and $m=n+\sum_{ij}m_{ij},$ then we have that $f$ has no $\ell^{1}$ inverse and $\lambda(f)$ is not invertible, so we really have to use formal $\ell^{2}$ inverses.
\end{example}
Suppose that
\[\begin{CD}  1 @>>> H @>\iota>> G @>\pi>> K@>>> 1,\end{CD}\]
is an exact sequence of groups. Then if $H,K$ are equipped with left-invariant partial orders, we can equip $G$ with a left-invariant partial order as follows. We say that $g_{1}\preceq g_{2}$ if either:
\begin{itemize}
    \item $\pi(g_{1}),\pi(g_{2})$ are comparable and $\pi(g_{1})\prec \pi(g_{2}),$ or
    \item $g_{1}^{-1}g_{2}\in H$ with $g_{1}^{-1}g_{2}\succeq 1.$
\end{itemize}

In particular, if $K,H$ are left-orderable, then $G$ is left-orderable.

\begin{example}
Fix an $n\in \N,n\geq 2$ and let $G=BS(1,n)=\ip{a,b:aba^{-1}=b^{n}}.$ There is a natural map $BS(1,n)\to \Z$ given by $a\mapsto 1,b\mapsto 0,$ and it is direct to see that the kernel is isomorphic to $\Z(1/n).$ Thus $G$ is left-orderable and there is a left-invariant ordering $<$ on $BS(1,n)$ which makes $a,b>1.$ So if we set
\[f=m+na+kb\]
with $m,n,k$ nonzero integers such that $|n|+|k|\leq m,$ then $G\actson (X_{f},m_{X_{f}})$ is isomorphic to a Bernoulli shift with entropy $\log(m).$
\end{example}

\begin{example}
Suppose that $G$ is \emph{polycyclic}. This means that we have a chain of groups
\[\{1\}=G_{0}\triangleleft G_{1}\triangleleft G_{2}\cdots G_{n-1}\triangleleft G_{n}=G\]
so that $G_{i}/G_{i-1}\cong \Z$ for all $1\leq i\leq n.$ Suppose we choose $x_{i}\in G_{i}$ so that $G_{i}=x_{i}G_{i-1}.$ Then it is possible to find a left-invariant order $<$ on $G$ so that $x_{i}>1$ for all $i.$ So if we choose $(n_{i})_{i=1}^{k}\in (\Z\setminus\{0\})^{k},$ $m\in \Z$ with $\sum_{i=1}^{k}|n_{i}|<m,$ then $f=m+\sum_{i}n_{i}x_{i}$ has $G\actson (X_{f},m_{X_{f}})$ isomorphic to a Bernoulli shift with entropy $\log(m).$ If $G$ is either superpolynomial growth, or polynomial growth of degree at least $5,$ then we can allow $\sum_{i=1}^{k}|n_{i}|\leq m$ to force  $G\actson (X_{f},m_{X_{f}})$ to be isomorphic to a Bernoulli shift with entropy $\log(m).$

It often happens that $G$ has exponential growth. For example, suppose that $A\in SL_{2}(\Z)$ has no eigenvalues on the unit circle. Then $\Z^{2}\rtimes_{A}\Z$ is of exponential growth. Let $a=(e_{1},0),b=(e_{2},0),c=(0,1)$ then for any $l,n,k,m\in \Z\setminus\{0\}$ with $|n|+|k|\leq m,$ setting $f=m-(la+nb+kc)$ we have that $G\actson (X_{f},m_{X_{f}})$ is isomorphic to a Bernoulli shift with entropy $\log(m).$

\end{example}

Suppose $X$ is a Polish space, and that $\preceq$ is a partial order on $X$ so that $\{y\in X:y\preceq x\}$ is closed for every $x\in X.$ Suppose that $G\actson X$ faithfully by order-preserving homeomorphisms. Then we can define a left-invariant partial order on $G$ as follows: let $(x_{n})_{n}$ be a dense sequence in $X,$ we then say that $g\prec h$ if
\begin{itemize}
    \item $\{n:gx_{n},hx_{n} \textnormal{ are comparable and not equal }\}$ is not empty,
    \item if $m=\min\{n:g x_{n},hx_{n}\textnormal{ are comparable and not equal}\},$ then $g x_{m}\prec h x_{m}.$
\end{itemize}
If $\preceq$ is a total order on $X,$ then it is easy to check that this gives a total order on $G.$

It is a folklore result that this construction characterizes left-orderable groups: namely, a countable group is left-orderable if and only if it embeds into the group of order-preserving homeomorphisms of $\R.$ See, for example, \cite[Theorem 6.8]{GhysOrder}.

The order described above seems fairly abstract. However, since we are allowed to prescribe the first few terms of our sequence $(x_{n})_{n},$ it makes it relatively straightforward to construct orders which make certain generators bigger than $1$ in that order.

\begin{example}
All the preceding examples were solvable. We can consider non-solvable examples, in fact groups of intermediate growth. Let $a,b,c,d$ be the homeomorphisms of $\Z^{\N}$ defined recursively by:
\[a(x_{1},x_{2},x_{3},\cdots)=(1+x_{1},x_{2},\cdots)\]
\[b(x_{1},x_{2},\cdots)=\begin{cases}
(x_{1},a(x_{2},x_{3},\cdots)),& \textnormal{ if $x_{1}$ is even}\\
(x_{1},c(x_{2},x_{3},\cdots)),& \textnormal{ if $x_{1}$ is odd}
\end{cases}\]
\[c(x_{1},x_{2},\cdots)=\begin{cases}
(x_{1},a(x_{2},x_{3},\cdots)),& \textnormal{if $x_{1}$ is even}\\
(x_{1},d(x_{2},x_{3},\cdots)),& \textnormal{ if $x_{1}$ is odd}
\end{cases}\]
\[d(x_{1},x_{2},\cdots)=\begin{cases}
(x_{1},x_{2},x_{3},\cdots),& \textnormal{if $x_{1}$ is even}\\
(x_{1},b(x_{2},x_{3},\cdots)),& \textnormal{ if $x_{1}$ is odd}
\end{cases}.\]
Let $G$ be the group generated by $a,b,c,d.$ Notice that $G$ preserves the lexographic ordering $<$ on $\Z^{\N}.$ Fix a dense sequence $(a_{n})_{n=1}^{\infty}$ in $\Z^{\N}$ with $a_{1}=(0,0,\cdots,0), a_{2}=(1,0,\cdots,0).$ Define an order $<$ on $G$ by saying that $g<h$ if when we set $n=\min\{l\in \N:g(a_{l})\ne h(a_{l})\},$ then $g(a_{n})<h(a_{n}).$ Then $<$ is a left-invariant order on $G$ which makes $a,b,c,d$ all bigger than $1.$ Also by \cite{GrigorIntGrowth3} this group has intermediate growth. So for any $m,m_{1},m_{2},m_{3},m_{4}\in \Z\setminus\{0\}$ with $\sum_{j=1}^{4}|m_{j}|\leq m,$ and with $f=m+m_{1}a+m_{2}b+m_{3}c+m_{4}d,$ we have that $G\actson (X_{f},m_{X_{f}})$ is a Bernoulli shift with entropy $\log(m).$

This group was defined in \cite{GRMakiOrder}, and its orderability was first shown there. We have followed the exposition in \cite[Section 1.1]{NavasOrderDiffeo}.
\end{example}

Our ability to use partial orders is in fact nontrivial, and we  can construct examples where the acting group is not torsion-free (and thus not left-orderable). For this, it will be helpful to switch to positive semigroups instead of partial orders.

Suppose $H$ is a group equipped with a left-invariant partial order $\preceq,$  and let $P$ be the corresponding positive semigroup. Suppose that $K\actson H$ by automorphisms  and that $K\cdot P\subseteq P.$ We can then define a positive semigroup $P_{G}$ in $G=H\rtimes K$ by  $P_{G}=\{(h,k):h\in P,k\in K\}.$ From this positive semigroup we get a left-invariant partial order as described before. In many cases, $K$ is finite and so $G$ \emph{is not} left-orderable.

A particular example is the case of \emph{generalized} wreath products. For a group $H,$ and a set $I,$ we will use $H^{\oplus I}=\{h\in H^{I}:h(i)=1\mbox{ for all but finitely many $i$}\}.$ Let $K$ be a group acting on a set $I,$ and $H$ another group. We then let $K\actson H^{\oplus I}$ by permuting the coordinate of $H$ (using the action $H\actson I$). The generalized wreath product $H\wr_{I}K$ is then the semidirect product $H^{\oplus I}\rtimes K.$ Suppose that $H$ has a left-invariant partial order $\preceq.$ We may then define a partial order $\preceq$ on $H^{\oplus I}$ by saying that if $h=(h_{i})_{i\in I},h'=(h'_{i})_{i\in I}\in H^{\oplus I},$ then $h\preceq k$ if and only if $h_{i}\preceq h'_{i}$ for all $i\in I.$ This is clearly invariant under the action of $K$ on $H$ for all $i\in I,$ and thus induces an order on $H\wr_{I}K$ by the above construction. Since this construction sometimes produces groups which are not torsion-free, we need to take some care in applying Proposition \ref{P:free for free} to have our actions be essentially free. The following lemma will do most of the work for us. This lemma is surely well known, but we will include a proof for completeness.

\begin{lem}\label{L:no finite normal subgroups}
Suppose that $H$ is an infinite group with no nontrivial finite normal subgroups. Let $I$ be a set, and let $K$ be a group with $K\actson I$ faithfully. Then $H\wr_{I}K$ has no nontrivial finite normal subgroups.
\end{lem}

\begin{proof}
We use $\alpha_{k}$ for the action of $k\in K$ on $H^{\oplus I}.$
Let $G=H\wr_{I}K,$ and suppose that $N$ is a finite normal subgroup in $G.$ Let $g\in N,$ and write $g=(h,k)$ with $h\in H^{\oplus I},$ and $k\in K.$ Let $C$ be the intersection of $H^{\oplus I}$ and the centralizer of $g$ in $G.$ Since $N$ is a finite normal subgroup, we know that the centralizer of $g$ in $G$ has finite index, and thus $C$ is a finite index subgroup of $H^{\oplus I}.$ Suppose that $(h',1)\in C,$ then $(h',1)(h,k)=(h'h,k)$ and $(h,k)(h',1)=(h\alpha_{k}(h'),k).$ So we are forced to have $\alpha_{k}(h')=h^{-1}h'h^{-1}.$ Hence $C=\{h'\in H^{\oplus I}:\alpha_{k}(h')=h^{-1}h'h^{-1}\}.$ Since $H$ is infinite, the fact that $C$ is finite-index in $H^{\oplus I}$ forces $\alpha_{k}$ to  act trivially on $I.$ Since the action of $K$ is faithful, we must have that $k=1.$ Thus $g\in N\cap H^{\oplus I}.$ Since $H$ has no nontrivial finite normal subgroups we must have that $N\cap H^{\oplus I}=\{1\}.$ Thus $g=1,$ and since $g$ was an arbitrary element of $N,$ we must have that $N$ is trivial.

\end{proof}

We will use this to give one more example of a Bernoulli shift where the acting group is torsion-free, and then one more where it is not.

\begin{example}
Consider $G=\Z\wr \Z,$ and use the natural order on $\Z$ to induce a left-invariant partial order as described above. Let $a'\in \Z^{\oplus Z}$ be given by $(a')_{n}=\delta_{n=0}$ and let $b=(0,1).$ Then $a,ab$ generate $G,$ and since $\Z^{k}$ embeds into $\Z\wr \Z$ for all $k,$ it is clear that $G$ has superpolynomial growth. Moreover, $aba^{-1}(a^{-1}(ab))^{k-1}$ is an infinite order element, and so $\ip{(ab)a^{-1},a^{-1}(ab)}$ is infinite. The left-invariant partial order describe above has $a\succ 1,ab\succ 1.$ Hence, for all $n,k,m\in \Z\setminus\{0\}$ with $|n|+|k|\leq m$ and with $f=m+ka+nab,$ we have that $G\actson (X_{f},m_{X_{f}})$ is a Bernoulli shift with entropy $\log(m).$
\end{example}

\begin{example}\label{E:hard wreath product example}
Fix $k\in \N,$ and let $G=\Z\wr (\Z/k\Z).$ Give $\Z$ its natural order and use this to induce an order on $G$ as described above.  Let $a'\in \Z^{\oplus (\Z/k\Z)}$ be given by $(a')_{n}=\delta_{n=0}$ and let $b=(0,1).$ Then $a,ab$ generate $G.$ Consider $n,l,m\in \Z\setminus\{0\}$ with $|n|+|l|<m$ if  $1\leq k\leq 4,$ and $|n|+|l|\leq m$ if $k\geq 5.$ Let $f\in \Z(G)$ be given by $f=m+na+lab.$ By Lemma \ref{L:no finite normal subgroups}, we know that $G$ has no finite normal subgroups, and thus by Proposition \ref{P:free for free} that $G\actson (X_{f},m_{X_{f}})$ is essentially free. So by Ornstein theory, we know that $G\actson (X_{f},m_{X_{f}})$ is a Bernoulli shift with entropy $\log(m).$
\end{example}

Let $S_{k}$ act on $\{1,\cdots,k\}$ in the natural way. In $G_{0}=\Z\wr_{\{1,\cdots,k\}} S_{k},$ one has a left-invariant partial order $\preceq$ so that  $\{g\in G_{0}:g\succ 1\}=\{(x,\sigma):x\in (\N\cup\{0\})^{k},x\ne 0,\sigma\in S_{k}\}.$ So if $H\leq S_{k}$ is arbitrary, we may consider $G=\Z^{k}\rtimes H $ as a group with a left-invariant partial order. So we can obtain similar modifications of Example \ref{E:hard wreath product example}. Let $S$ be a set of generators for $H,$ and let $\mathcal{O}\subseteq (\N\cup\{0\})^{k}\setminus\{0\}$ be such that the group generated by $H\mathcal{O}$ is all of $\Z^{k}.$ Let $(m_{x})_{x\in O},(l_{y,s})_{(y,s)\in \mathcal{O}\times S},m\in (\Z\setminus\{0\})^{k}$ be such that $\sum_{s}|m_{s}|+\sum_{s,o}|l_{s,o}|\leq m.$ Set
\[f=m+\sum_{x}m_{x}(x,\id)+\sum_{(y,s)\in \mathcal{O}\times S}l_{y,s}(y,s)\in \Z(G).\]
If $k\geq 5,$ then  $G\actson (X_{f},m_{X_{f}})$ is a Bernoulli shift with entropy $\log(m).$ If $k\leq 4,$ then as long as we assume that $\sum_{s}|m_{s}|+\sum_{s,o}|l_{s,o}|<m,$ we still have that $G\actson (X_{f},m_{X_{f}})$ is a Bernoulli shift with entropy $\log(m).$

Of course the possibilities here are endless, and one can consider other wreath products $H\wr_{I}K$ where $H$ is left-orderable. E.g. one  can take $H$ to be the Heisenberg group, or other polycyclic groups.

\subsection{Nonamenable examples}

\begin{example}
For an integer $r>1,$ let $\F_{r}$ be the free group on letters $\{a_{1},\cdots,a_{r}\}.$ Let $P$ be the set of nonidentity elements of $\F_{r}$ whose word decompositions only have positive powers of the generators,  then $P$ is a positive semigroup. Thus $P$ induces a left-invariant order $\preceq$ on $\F_{r}$ by $g\preceq h$ if $g^{-1}h\in P\cup\{1\}.$ In fact, by \cite{OrderedFreeProducts} we know that there is a left-invariant total order $<$ on $\F_{r}$ so that $P\subseteq \{g\in \F_{r}:g>1\},$ but we will not need this. Thus if $(m_{j})_{j=1}^{r}\in(\Z\setminus\{0\})^{r},m\in\N$ with $\sum_{j=1}^{r}|m_{j}|\leq m,$ and $f=m+\sum_{j=1}^{r}m_{j}a_{j},$ then $G\actson (X_{f},m_{X_{f}})$ is a factor of a Bernoulli shift.
\end{example}

More generally, any residually free group  is left-orderable. This includes fundamental groups of compact surfaces without boundary whose genus is larger than $1$. Additionally, if $G$ is residually free, and we have an \emph{explicit} family of homomorphisms $\pi_{n}\colon G\to \F_{r(n)}$ for integers $r(n)\in\N$ so that $\bigcap_{n}\ker(\pi_{n})=\{1\},$ then we get an explicit left-invariant partial order on $G.$ Thus in many cases, we can explicitly describes semi-lopsided elements $f\in \Z(G)$ so that $G\actson (X_{f},m_{X_{f}})$ is a factor of a Bernoulli shift.
More generally by \cite{OrderedFreeProducts}, we also have that free products of left-orderable groups are left-orderable.

\begin{example}
For an integer $n\geq 3,$ consider the \emph{braid group} $B_{n}$ which has the following presentation:
\[B_{n}=\ip{\sigma_{1},\cdots,\sigma_{n-1}:\sigma_{i}\sigma_{i+1}\sigma_{i}=\sigma_{i+1}\sigma_{i}\sigma_{i+1},\mbox{ for $1\leq i\leq n-2$ and } \sigma_{i}\sigma_{j}=\sigma_{j}\sigma_{i}\mbox{ if $|i-j|\geq 2$}}.\]
Dehornoy proved (see \cite{Dehornoy}) that $B_{n}$ has a left-invariant order $<$ which is now called the \emph{Dehornoy order}. This ordering is uniquely defined by saying that for all $1\leq i\leq n-2$ we have $\beta_{0}\sigma_{i}\beta_{1}>1$ for all $\beta_{0},\beta_{1}\in \ip{\sigma_{i+1},\cdots,\sigma_{n-1}}.$ Let $m,m_{1},\cdots,m_{n-1}\in \Z\setminus\{0\},$ and $f=m+\sum_{j=1}^{n-1}m_{j}\sigma_{j}.$ Since $n\geq 3,$ we know that $B_{n}$ contains a free group on two generators and thus has exponential growth. So if $\sum_{j}|m_{j}|\leq |m|,$ we have that $G\actson (X_{f},m_{X_{f}})$ is a factor of a Bernoulli shift.

\end{example}

More generally, let $S$ be a compact surface with a finite set of punctures (potentially empty) and nonempty boundary, and let $G$ be the mapping class group of $S.$ Then by \cite{OrderMCG}, we know that $G$ is left-orderable (see also \cite{GeomOrderMCG}). In many cases, we can explicitly describe a left-invariant order on $G$ and as before this allows us to explicitly produces semi-lopsided $f\in\Z(G)$ with $G\actson (X_{f},m_{X_{f}})$ a factor of a Bernoulli shift.

\begin{example}
Consider Thompson's group $F$ which is the group of all increasing, piecewise linear homeomorphisms of $[0,1]$ whose break points are dyadic rationals, and whose slopes are powers of $2.$ By definition, $F$ is a subgroup of the group of increasing homeomorphisms of $[0,1]$ and is thus left-orderable. Let $x_{0}$ denote the element of $F$ whose break points are $\frac{1}{4},\frac{1}{2}$ and has $x_{0}(\frac{1}{4})=\frac{1}{2},x_{0}(\frac{1}{2})=\frac{3}{4}.$ Let $x_{1}$ be the element of $F$ whose break points are $\frac{1}{2},\frac{3}{4},\frac{7}{8}$ with $x_{1}(\frac{1}{2})=\frac{1}{2},x_{1}(\frac{5}{8})=\frac{3}{4},x_{1}(\frac{3}{4})=\frac{7}{8}.$ It is known (see \cite[Section 3]{NotesThompson}) that $x_{0},x_{1}$ generate $F,$ and that $F$ has exponential growth. We may choose a left-invariant order $<$ so that $x_{0},x_{1}>1,$ for example by considering a dense sequence $(t_{n})_{n=1}^{\infty}$ in $[0,1]$ with $t_{1}=\frac{5}{8}$ and using this to define a left-invariant order on $F$ as described before. Thus if $n,l,m\in \Z\setminus\{0\}$ with $|n|+|l|\leq m,$ then $f=m+nx_{0}+lx_{1}$ has $G\actson (X_{f},m_{X_{f}})$ a factor of a Bernoulli shift.
\end{example}

\begin{example}
Let $G=\Z/k\Z*\Z/k\Z$ for $k\geq 3.$ Let $x$ be the generator of the first factor of $\Z/k\Z,$ and let $y$ be the generator of the second factor. Let $P$ be the semigroup generated by $xy,x^{2}y^{2}.$ A simple exercise shows that this is a positive semigroup with $\ip{P}=G,$ and that $(x^{2}y^{2})^{-1}xy$ has infinite order. Thus if $m,n,l\in \Z\setminus\{0\}$ and $|n|+|l|\leq m,$ then $f=m+nxy+lx^{2}y^{2},$ then $G\actson (X_{f},m_{X_{f}})$ is a factor of a Bernoulli shift.

In this case, it is also direct to establish that $G$ has no finite normal subgroups. So, we know that $G\actson (X_{f},m_{X_{f}})$ is free. This is of less significance in this case, as Ornstein theory no longer applies in the nonamenable setting.

\end{example}

\begin{example}
Fix an integer $k>1.$ Regard $\F_{k}$ as the free group on letters $a_{1},\cdots,a_{k}.$ Let $P$ be the semigroup generated by $a_{1},\cdots,a_{k}.$  As before we have that $P$ is a positive semigroup in $\F_{k}.$ Consider the natural action of $S_{k}$ by automorphisms on $\F_{k}$ given by permuting the generators. This semigroup is clearly invariant under $S_{k},$ so this induces a left-invariant partial order on $G_{0}=\F_{k}\rtimes S_{k}.$

Regard $\Z/k\Z\leq S_{k}$ via the translation action on $\Z/k\Z.$ Let $G=\F_{k}\rtimes \Z/k\Z.$ Let $b=(0,1+k\Z).$ Then if $m,m_{1},\cdots,m_{k},n\in \Z\setminus\{0\}$ and $|n|+\sum_{j}|m_{j}|\leq m,$ then
\[f=m+n a_{1}b+\sum_{j}m_{j}a_{j}\]
is such that $G\actson (X_{f},m_{X_{f}})$ is a factor of a Bernoulli shift.

Again, in this case one can argue as in Lemma \ref{L:no finite normal subgroups} to show that $G$ (and also $G_{0}$) has no nontrivial finite normal subgroups. So $G\actson (X_{f},m_{X_{f}})$ is also essentially free in this case.
\end{example}

\section{Closing Remarks}\label{S:closing}

Suppose that $G$ is a countable, discrete group and that $f\in \Z(G)$ is semi-lopsided. If $G$ is assumed sofic, then we know that the entropy of $G\actson (X_{f},m_{X_{f}})$ is $\log(\tau(f))$ from the results of \cite{Me5,LiThom}, and Appendix \ref{S:FKD}. It is worth noting that if $\tau(f)$ is odd, then the proof of Corollary \ref{C:factor of a Bernoulli shift main one} exhibits $X_{f}$ as a factor of a Bernoulli shift \emph{which has equal entropy}. It thus makes it very plausible that this factor map is, in fact, an isomorphism. If $\tau(f)$ is even, then the domain of $\Theta_{\xi}$ is a Bernoulli shift whose entropy is not equal to that of $G\actson (X_{f},m_{X_{f}}).$ So the factor map exhibited in the proof of Corollary \ref{C:factor of a Bernoulli shift main one} is not injective modulo null sets (if $G$ is assumed sofic). However,  under the stronger assumption that $f$ has an $\ell^{1}$ formal inverse we can exhibit a factor map from a Bernoulli shift with equal entropy.

To prove this, we will need to note that, though we did not prove this in \cite{MeMaxMinWC}, we can replace the assumption that $\nu$ has mean zero and finite second moment with the assumption that $\nu$ has a finite first moment \emph{provided we work  with $\ell^{1}$ vectors instead of $\ell^{2}$ vectors}.
This follows from the exact same methods as in \cite[Section 3]{MeMaxMinWC}. In this case there is no need to apply the uniform continuity as in \cite[Section 3]{MeMaxMinWC}, since the fact that $\nu$ has finite first moment implies that for $\nu^{\otimes G}$-almost every $x\in \R^{G}$ it is true that for every $\xi\in \ell^{1}(G,\R),g\in G$ the series
\[\sum_{h\in  G}x(h)\xi^{*}(h^{-1}g)\]
converges absolutely. We state the analogous version of Theorem \ref{T: background convolution} for convolving with $\ell^{1}$ vectors here. The only difficult part is computing the Fourier transform of $(\Theta_{\xi})_{*}(\nu^{\otimes G}),$ and this follows by identical arguments as those in \cite{MeWE, MeMaxMinWC}.

\begin{thm}\label{T:ell1 convolution pushforward}
Let $G$ be a countable, discrete group and fix a $\nu\in \Prob(\R)$ with finite first moment. For $\xi\in \ell^{1}(G,\R)$ let $\Theta_{\xi}\colon \R^{G}\to \T^{G}$ be the (almost everywhere defined) map $\Theta_{\xi}(x)=q(x*\xi^{*})+\Z.$ Set $\mu_{\xi}=(\Theta_{\xi})_{*}(\nu^{\otimes G}).$ Then for every $\alpha\in \Z(G),$
\[\widehat{\mu}_{\xi}(\alpha)=\prod_{g\in G}\widehat{\nu}((\alpha \xi)(g)).\]
\end{thm}

Using this, we can exhibit $G\actson (X_{f},m_{X_{f}})$ as an equal entropy factor of a Bernoulli shift, provided that $f$ has an $\ell^{1}$ formal inverse.

\begin{cor}\label{C:equal entropy factor map}
Let $G$ be a countable, discrete group and let $f\in \Z(G).$ Let $S=\supp(\widehat{f})\setminus \{1\},$ and $H=\ip{S}.$ Suppose that there is a left-invariant partial order $\preceq$ on $H$ so that $S\subseteq \{h\in H:h\succ 1\}.$ Set $m=\tau(f).$
\begin{enumerate}[(i)]
\item Assume that $m=2k+1$ is odd and that $f$ has an $\ell^{2}$ formal inverse $\xi.$ Set $\nu=u_{\{-k,\cdots,k\}},$ and let $\Theta_{\xi}$ be defined as in  Theorem \ref{T: background convolution} corresponding to this $\nu.$ Then $(\Theta_{\xi})_{*}(\nu^{\otimes G})=m_{X_{f}}.$ If $G$ is assumed sofic, then the actions $G\actson (\{-k,\cdots,k\}^{G},u_{\{-k,\cdots,k\}}^{\otimes G}), G\actson (X_{f},m_{X_{f}})$ both have entropy $\log(m)$ for any sofic approximation of $G.$ \label{I:ell2 inverse case closing remarks}
\item Assume that $f$ has an $\ell^{1}$ formal inverse $\xi.$ Define $\Theta_{\xi}\colon \{1,\cdots,m\}^{G}\to X_{f}$ by $\Theta_{\xi}(x)(g)=(x\xi^{*})(g)+\Z.$ Then $(\Theta_{\xi})_{*}(u_{m}^{\otimes G})=m_{X_{f}}.$ If $G$ is assumed sofic, then the actions $G\actson (\{-k,\cdots,k\}^{G},u_{\{-k,\cdots,k\}}^{\otimes G}), G\actson (X_{f},m_{X_{f}})$ both have entropy $\log(m)$ for any sofic approximation of $G.$ \label{I:ell1 inverse case closing remarks}
\end{enumerate}

\end{cor}

\begin{proof}

(\ref{I:ell2 inverse case closing remarks}): The fact that $(\Theta_{\xi})_{*}(\nu^{\otimes G})=m_{X_{f}}$ is shown in the course of the proof of Corollary \ref{C:factor of a Bernoulli shift main one}. By \cite{Bow}, we know that $G\actson (\{-k,\cdots,k\}^{G},u_{\{-k,\cdots,k\}}^{\otimes G})$ has entropy $\log(m)$ with respect to any sofic approximation of $G.$ The fact that $G\actson (X_{f},m_{X_{f}})$ has entropy $\log(m)$ is a consequence of \cite{LiThom, Me5} and the fact that the Fuglede-Kadison determinant of $f$ is $\log(m)$ (see Appendix \ref{S:FKD}, namely Corollary \ref{C:FKD Calculation}, for more details).

(\ref{I:ell1 inverse case closing remarks}): Set $\mu=(\Theta_{\xi})_{*}(u_{m}^{\otimes G}).$ By Theorem \ref{T:ell1 convolution pushforward}, for $\alpha\in \Z(G),$
\[\widehat{\mu}(\alpha)=\prod_{g\in G}\widehat{\nu}((\alpha\xi)(g)).\]
Since $\widehat{\nu}(t)=1$ for $t\in \Z$ and $\widehat{\nu}(t)=0$ for every $t\in \frac{1}{m}\Z\cap \Z^{c},$ it follows as in the proof of Corollary \ref{C:factor of a Bernoulli shift main one} that $\mu=m_{X_{f}}.$

\end{proof}

Thus in either case (\ref{I:ell2 inverse case closing remarks}), (\ref{I:ell1 inverse case closing remarks}) of Corollary \ref{C:equal entropy factor map} we can exhibit $G\actson (X_{f},m_{X_{f}})$ as an equal factor map of a Bernoulli shift with the same entropy, so this makes it plausible that this factor map is an isomorphism.

\begin{conj}\label{con:injectivity conjecture}
Let $G$ be a countable, discrete, group with a left-invariant partial order $\preceq.$ Let $f\in \Z(G)$ be semi-lopsided and such that $\supp(\widehat{f})\setminus\{1\}\subseteq \{g\in G:g\succ 1\}.$ Set $m=\tau(f).$ Suppose that either:
\begin{enumerate}[(i)]
    \item $f$ has an $\ell^{2}$ formal inverse $\xi$ and $m$ is odd,
    \item $f$ has an $\ell^{1}$ formal inverse.
\end{enumerate}
In case (\ref{I:ell2 inverse case closing remarks}) let $\Theta_{\xi}$ be defined as in case (\ref{I:ell2 inverse case closing remarks}) of Corollary \ref{C:equal entropy factor map}, and in case (\ref{I:ell1 inverse case closing remarks}) let $\Theta_{\xi}$ be defined as in case (\ref{I:ell1 inverse case closing remarks}) of Corollary \ref{C:equal entropy factor map}. Then $\Theta_{\xi}$ is injective modulo null sets.
\end{conj}

We remark that if $G$ is a free group, then Lind-Schmidt \cite{LindSchmidtFreeBern} can show that case (\ref{I:ell1 inverse case closing remarks}) of Conjecture \ref{con:injectivity conjecture} is true for some examples of $f.$

We also remark here that in the proof of Corollary \ref{C:factor of a Bernoulli shift main one}, we did not really need to find a fixed choice of $m\in \N$ so that for every $\alpha\in \Z(G)\setminus \Z(G)f$ there is some $g_{0}\in G$ with $\widehat{\nu}((\alpha\xi)(g_{0}))\in\frac{1}{m}\Z\cap \Z^{c}.$ In actuality the major point here is that we need to ensure that if $\alpha\in\Z(G)\setminus \Z(G)f,$ then there is some $g_{0}\in G$ so that  $\widehat{\nu}((\alpha\xi)(g_{0}))$ is a noninteger rational  number whose denominator is ``not too big." We state this precisely as follows.

\begin{thm}
Let $G$ be a countable, discrete group and $f\in \Z(G).$ Suppose that $f$ has an  $\ell^{2}$ formal inverse $\xi.$ Suppose that there is an $M\in\N$ with the following property. For every $\alpha\in \Z(G)\cap (\Z(G)f)^{c},$ there is a $g_{0}\in G$ so that $(\alpha\xi)(g_{0})\in \frac{1}{k}\Z\cap \Z^{c}$ for some $1\leq k\leq M.$ Then $G\actson (X_{f},m_{X_{f}})$ is a factor of a Bernoulli shift.
\end{thm}

\begin{proof}
It suffices to find a probability measure $\nu\in \Prob(\Z)$ which is finitely supported and has mean zero and so that $\widehat{\nu}(t)=0$ for every $t\in \Z^{c}\cap \bigcup_{k=1}^{M}\frac{1}{k}\Z.$ We can then simply follow the proof of Corollary \ref{C:factor of a Bernoulli shift main one} to see that $G\actson (X_{f},m_{X_{f}})$ is a factor of a Bernoulli shift. For $1\leq k\leq m,$ let
\[\nu_{k}=\begin{cases}
u_{\{-l,\cdots,l\}}& \textnormal{ if $k=2l+1$ is odd,}\\
u_{\{-k,\cdots,-1\}}*u_{\{1,\cdots,k\}}& \textnormal{if $k$ is even.}
\end{cases}\]
Now set $\nu=\nu_{1}*\nu_{2}*\cdots*\nu_{M}.$ As
\[\widehat{\nu}=\prod_{k=1}^{M}\widehat{\nu}_{k},\]
and
\[\int x\,d\nu(x)=\sum_{k=1}^{M}\int x\,d\nu_{k}(x),\]
it is easy to see that $\nu$ has the desired properties.

\end{proof}

 \appendix

\section{Background on tracial von Neumann algebras}\label{S:Tracial vNa}

Let $\mathcal{H}$ be a Hilbert space. Every element of $B(\mathcal{H})$ is a function $\mathcal{H}\to \mathcal{H}$ and so we may view $B(\mathcal{H})\subseteq \mathcal{H}^{\mathcal{H}}.$ The \emph{strong operator topology} on $B(\mathcal{H})$ is the subspace topology inherited from the product topology on $\mathcal{H}^{\mathcal{H}}$. We may also define this topology by prescribing a basis. Given $T\in B(\mathcal{H}),$ a finite $F\subseteq \mathcal{H},$ and an $\varepsilon>0,$ set
\[U_{F,\varepsilon}(T)=\bigcap_{\xi\in F}\{S\in B(\mathcal{H}):\|(S-T)\xi\|<\varepsilon\}.\]
The collection $(U_{F,\varepsilon})_{F,\varespilon}$ ranging over all finite $F\subseteq \mathcal{H}$ and $\varepsilon>0$ form a neighborhood basis at $T$ in the strong operator topology.

\begin{defn}
A \emph{von Neumann algebra} is a subalgebra $M\subseteq B(\mathcal{H})$ for some Hilbert space $\mathcal{H}$ which is closed under adjoints and in the strong operator topology.

A \emph{tracial von Neumann algebra} is a pair $(M,\tau)$ where $M$ is a von Neumann algebra and $\tau\colon M\to \C$ is a linear functional which is:
\begin{itemize}
    \item \emph{a state:} $\tau(x^{*}x)\geq 0$ for all $x\in M,$ and $\tau(1)=1$,
    \item \emph{faithful:} $\tau(x^{*}x)=0$ if and only if $x=0$,
    \item \emph{tracial:} $\tau(xy)=\tau(yx)$ for all $x,y\in M,$
    \item \emph{normal:} $\tau\big|_{\{x\in M:\|x\|\leq 1\}}$ is strong operator topology continuous.
\end{itemize}
\end{defn}

For notation, we will usually use $1$ for the identity operator on $B(\mathcal{H}).$
The main example which concerns us is the \emph{group von Neumann algebra}. Let $G$ be a countable discrete group, and define $\lambda\colon G\to \mathcal{U}(\ell^{2}(G))$ by
\[(\lambda(g)\xi)(h)=\xi(g^{-1}h)\mbox{ for $g,h\in G,\xi\in \ell^{2}(G).$}\]
The group von Neumann algebra of $G,$ denoted $L(G),$ is defined by
\[L(G)=\overline{\Span\{\lambda(g):g\in G\}}^{SOT}.\]
Here the span and closure are taken by viewing $\mathcal{U}(\ell^{2}(G))\subseteq B(\mathcal{H}).$
Define $\tau\colon L(G)\to \C$ by
\[\tau(x)=\ip{x\delta_{1},\delta_{1}}.\]
We leave it as an exercise to check that $(L(G),\tau)$ is a tracial von Neumann algebra. Another good case for intuition is the pair $(L^{\infty}(X,\mu),\int \cdot \,d\mu)$ for a probability space $(X,\mu).$ We may view $L^{\infty}(X,\mu)$ as a von Neumann algebra by identifying an essentially bounded function with its multiplication operator acting on $L^{2}(X,\mu).$ It turns out that every tracial von Neumann algebra with $M$ abelian is of the form $(L^{\infty}(X,\mu),\int \cdot \,d\mu)$ for some probability space $(X,\mu)$.

We may use the trace to construct a representation of $M$ which may be nicer than the original Hilbert space on which $M$ is represented by following the abelian case. We define an inner product on $M$ by $\ip{x,y}=\tau(y^{*}x).$ We use $\|\cdot\|_{2}$ for the resulting norm given by $\|x\|_{2}=\tau(x^{*}x)^{1/2}$ for $x\in M,$ and we let $L^{2}(M,\tau)$ be the completion of $M$ under this inner product. We will still use $\|\cdot\|_{2}$ for the norm on $L^{2}(M,\tau).$ For $x\in M,$ we sometimes use $\hat{x}$ for $x$ viewed as a vector in the completion $L^{2}(M,\tau).$ For $x,y\in M$ we have
\[\|xy\|_{2}\leq \|x\|\|y\|_{2},\]
\[\|yx\|_{2}\leq \|y\|_{2}\|x\|,\]
(see \cite[Chapter V.2, Equation 8]{Taka})
where $\|x\|$ is the operator norm of $x$ acting on the original Hilbert space $\mathcal{H}$ that $M$ is defined on.
It follows from the above estimates that left/right multiplication on $M$ are $\|\cdot\|_{2}$-$\|\cdot\|_{2}$ uniformly continuous as maps $M\to M$ and thus have continuous extensions to bounded operators on $L^{2}(M,\tau).$ For $\xi\in L^{2}(M,\tau)$ we will use $x\xi $, $\xi x$ for the image of $\xi$ under the continuous extension of the left/right multiplication operators given above. It is a consequence of faithfulness of $\tau$ (see \cite[Theorem VIII.4.8]{Conway}) that the operators $\xi\mapsto x\xi,$ $\xi\mapsto \xi x$ have operator norm equal to operator norm of $x$ acting on the original Hilbert space $\mathcal{H}$ that $M$ is defined on.
In analogy with the abelian case, we will use $\|x\|_{\infty}$ for the operator norm of $x.$

This construction is relatively concrete in the group case. If $M=L(G)$ and $\tau=\ip{ \cdot \delta_{1},\delta_{1}},$ then it is direct to show that
\[\ip{\lambda(\alpha)\delta_{1},\lambda(\beta)\delta_{1}}=\ip{\widehat{\lambda(\alpha)},\widehat{\lambda(\beta)}}\]
for all $\alpha,\beta\in \C(G).$ From this it follows that there is unique unitary $U\colon \ell^{2}(G)\to L^{2}(M,\tau)$ which sends $\delta_{1}$ to $\hat{1}$ and which is equivariant with respect to the natural action of $M$ on $\ell^{2}(G)$ and the left multiplication action of $M$ on $L^{2}(M,\tau).$

One of the most important features of von Neumann algebras is the ability to apply bounded Borel functions, and not just any polynomial, to (certain) elements of $M.$ For concreteness we will stick to the self-adjoint case, but this works more generally for any normal element in a von Neumann algebra. If $E\subseteq \C,$ we will use $\Bor^{\infty}(E)$ for the algebra of bounded, Borel, $\C$-valued functions on $E.$ This is an algebra under pointwise multiplication, and it is also a $*$-algebra using $f^{*}=\overline{f}.$ For $f\in \Bor^{\infty}(E),$ we use $\|f\|$ for the supremum norm of $f.$  We now recall the notion of Borel functional calculus, as well as the spectral theorem for self-adjoint operators. For $x\in M,$ we define its spectrum by
\[\sigma(x)=\{\lambda\in \C:x-\lambda 1 \mbox{ is not invertible}\},\]
by \emph{invertible} we either mean in the ring $M$ or the ring $B(\mathcal{H})$ (these notions are equivalent for von Neumann algebras by \cite[Proposition VIII.1.14]{Conway}). If $x\in M$ is self-adjoint, then $\sigma(x)\subseteq \R$ (see \cite[Corollary VII.1.13]{Conway}). In fact, for $x\in M$ self-adjoint we have that $\sigma(x)\subseteq [-\|x\|_{\infty},\|x\|_{\infty}]$ (see \cite[Proposition VIII.1.11 (e)]{Conway}). We say that $x\in M$ is \emph{positive}, and write $x\geq 0$, if $\sigma(x)\subseteq [0,\infty).$ This ends up being equivalent to saying that $\ip{x\xi,\xi}\geq 0$ for all $\xi\in \mathcal{H}$ (see \cite[Theorem VII.3.8]{Conway}). If $x,y\in M$ are self-adjoint, we then say $x\leq y$ if $y-x\geq 0.$

\begin{thm}[Borel functional calculus, Theorem IX.8.10 of \cite{Conway}]
Let $\mathcal{H}$ be a Hilbert space, and $M\subseteq B(\mathcal{H})$ a von Neumann algebra. Then for every self-adjoint $x\in M$  there is a unique $*$-homomorphism $\pi\colon\Bor^{\infty}(\sigma(x))\to M$ satisfying the following axioms:
\begin{itemize}
    \item $\pi$ maps the identically $1$ function to $1\in M,$
    \item $\pi$ maps the function $t\mapsto t$ to $x,$
    \item $\|\pi(f)\|_{\infty}\leq \|f\|$ for every $f\in \Bor^{\infty}(\sigma(x))$,
    \item if $f_{n}$ is a sequence in $\Bor^{\infty}(E),$ with $\sup_{n}\|f_{n}\|<\infty,$ and if $f_{n}$ converges pointwise on $E$ to $f\in \Bor^{\infty}(E)$, then $\pi(f_{n})\to \pi(f)$ in the strong operator topology.
\end{itemize}
Moreover, $\pi$ satisfies the following properties:
\begin{itemize}
    \item $\sigma(\pi(f))=f(\sigma(x))$ for every $f\in C(\sigma(x)),$
    \item if $f,g\in\Bor^{\infty}(\sigma(x))$ and $f\leq g,$ then $\pi(f)\leq \pi(g).$
\end{itemize}
\end{thm}
For $x\in M$ and $f\in \Bor^{\infty}(\sigma(x))$ we will use $f(x)$ for $\pi(f).$ This agrees with the standard notation when $f$ is a polynomial. It follows from uniqueness that the Borel functional calculus behaves well with respect to composition: if $f\in \Bor^{\infty}(\sigma(x)),$ and $g\in \Bor^{\infty}(\overline{f(\sigma(x))}),$ then $g(f(x))=(g\circ f)(x).$
It is frequently of use to reduce the study of arbitrary elements of $M$ to the self-adjoint case. For $x\in M$ (not necessarily self-adjoint) we set $|x|=(x^{*}x)^{1/2}.$ Here $(\cdot)^{1/2}$ is interpreted in terms of functional calculus: we are applying the square root function defined on $[0,\infty)$ to the positive element $x^{*}x$. By direct computation, $\||x|\xi\|=\|x\xi\|$ for every $\xi\in \mathcal{H}.$ In particular, $\||x|\|_{\infty}=\|x\|_{\infty},$ and so
\[\sigma(|x|)\subseteq [-\|x\|_{\infty},\|x\|_{\infty}]\cap [0,\infty)=[0,\|x\|_{\infty}].\]
We also set $\Re(x)=\frac{x+x^{*}}{2}.$ We caution the reader that in the case of $L(G)$ these notions do \emph{not} agree with the applying the absolute value and real part operations to the coefficients. Namely, if $\alpha\in \C(G),$ then $\Re(\widehat{\alpha})\neq \left(\frac{\alpha+\alpha^{*}}{2}\right)^{\widehat{}}.$ Similarly, if $\beta\in \C(G)$ has $\widehat{\beta}=|\widehat{\alpha}|,$ then $\lambda(\beta)\neq |\lambda(\alpha)|$ (indeed $\lambda(\beta)$ is typically not even self-adjoint, much less positive).

Another important consequence of Borel functional calculus is (a version of) the Spectral Theorem. Note that if $E\subseteq\R$ is Borel, then $1_{E}(x)$ is self-adjoint and idempotent. This implies that $1_{E}(x)$ has closed image and is the orthogonal projection onto its image.
\begin{thm}[The Spectral Theorem]
Let $\mathcal{H}$ be a Hilbert space and $T\in B(\mathcal{H})$ self-adjoint. Then for every $\xi\in \mathcal{H}$ there is a unique Borel measure $\nu$ on $\sigma(T)$ with
\[\ip{f(T)\xi,\xi}=\int f\,d\nu,\mbox { for all $f\in \Bor^{\infty}(\sigma(T))$}.\]
Moreover, $\nu(\sigma(T))=\|\xi\|^{2}$.

\end{thm}
Alternatively, we can characterize $\nu$ as above by saying that $\nu(E)=\ip{1_{E}(T)\xi,\xi}=\|1_{E}(T)\xi\|^{2}$ for every Borel $E\subseteq \R.$ Perhaps more concretely, the fact that $\nu$ is compactly supported allows us to use the Stone-Weierstrass and Riesz representation theorems to say that $\nu$ is the unique measure satisfying
\[\int t^{n}\,d\nu(t)=\ip{T^{n}\xi,\xi}\mbox{ for every $n\in \N$.}\]
Of particular interest is the case $\xi=\hat{1}.$ In this case, we will use $\mu_{x}$ for the unique measure satisfying
\[\tau(f(x))=\ip{f(x)\hat{1},\hat{1}}=\int f\,d\mu_{x}\]
for all $f\in \Bor^{\infty}(\sigma(x)).$ The following is a well known result, but we isolate it because we will use it rather frequently in the appendix and because we are unable to find a short reference for it in the literature.

\begin{prop}\label{P:characterizing injectivity appendix}
Let $(M,\tau)$ be a tracial von Neumann algebra, and let $x\in M.$ Then the following are equivalent:
\begin{enumerate}
    \item $x$ is injective as an operator on $L^{2}(M,\tau)$, \label{I:injective as op app}
    \item $1_{\{0\}}(|x|)=0,$ \label{I:injective as zero proj app}
    \item $\mu_{|x|}(\{0\})=0.$ \label{I:injective as measure app}
\end{enumerate}

\end{prop}

\begin{proof}
(\ref{I:injective as op app}) implies (\ref{I:injective as zero proj app}): Suppose $x$ is injective as an operator on $L^{2}(M,\tau)$ and let $\xi\in L^{2}(M,\tau).$ Then
\[\|x1_{\{0\}}(|x|)\xi\|_{2}=\||x|1_{\{0\}}(|x|)\xi\|_{2}=0,\]
the last equality following as the function $t\mapsto t1_{\{0\}}(t)$ is identically zero. By injectivity of $x,$ it follows that $1_{\{0\}}(|x|)\xi=0.$ So we have shown that $1_{\{0\}}(|x|)=0.$

(\ref{I:injective as zero proj app}) implies (\ref{I:injective as measure app}): This follows from the fact that $\mu_{|x|}(\{0\})=\tau(1_{\{0\}}(|x|)).$

(\ref{I:injective as measure app}) implies (\ref{I:injective as zero proj app}): This follows from the fact that $\tau$ is faithful.

(\ref{I:injective as zero proj app}) implies (\ref{I:injective as op app}): Suppose $1_{\{0\}}(|x|)=0.$  Let $\xi\in \mathcal{H}$ with $\xi\ne 0.$ Let $\nu$ be the measure on $\sigma(|x|)$ with
\[\ip{f(|x|)\xi,\xi}=\int f\,d\nu,\mbox{ for all $f\in \Bor^{\infty}(\sigma(x))$}.\]
Since $1_{\{0\}}(|x|)=0,$ we have that $1=1-1_{\{0\}}(|x|)=1_{(0,\infty)}(|x|).$ Thus
\[\|x\xi\|_{2}^{2}=\ip{|x|^{2}\xi,\xi}=\ip{|x|^{2}1_{(0,\infty)}(|x|)\xi,\xi}=\int_{(0,\infty)}t^{2}\,d\nu(t).\]
Since $\nu(\{0\})=\ip{1_{\{0\}}(|x|)\xi,\xi}=0,$ and $\nu([0,\infty))=\|\xi\|_{2}^{2}>0,$ the above integral must be positive. So $x$ has trivial kernel, and is thus injective.

\end{proof}

\subsection{Noncommutative $L^{p}$-spaces and the Proof of Lemma \ref{L:sot convergence} }\label{S:nclp}
For the proof of Lemma \ref{L:sot convergence}, we will use the notion of noncommutative $L^{p}$-spaces.

\begin{defn}
Let $(M,\tau)$ be a tracial von Neumann algebra and $x\in M.$ For $1\leq p<\infty,$ we define the noncommutative $L^{p}$-norm of $x$ with respect to $\tau$ by
\[\|x\|_{p}=\tau(|x|^{p})^{1/p}.\]
\end{defn}
It is a non-obvious fact that $\|\cdot\|_{p}$ is indeed a norm on $M$ (see \cite[Theorem IX.2.13]{TakesakiII} or \cite[Theorem 2.1.6]{silva2018lecture}). Given $x\in M,$ we use $\|x\|_{\infty}$ for the operator norm of $x$ regarded as an operator on $L^{2}(M,\tau)$ (recall that this is the same as the operator norm of $x$ represented on the original Hilbert space $M$ is given on).

Moreover, we still have H\o lder's inequality in the noncommutative context. Namely, if $p,q,r\in [1,\infty]$ with $\frac{1}{p}+\frac{1}{q}=\frac{1}{r},$ then for $x,y\in M$ one has
\[\|xy\|_{r}\leq \|x\|_{p}\|y\|_{q}.\]
See \cite[Corollaire 3]{DixmierLp} or \cite[Theorem 2.1.5]{silva2018lecture} for a proof.

\begin{prop}\label{P:more general sot convergence}
Let $(M,\tau)$ be a tracial von Neumann algebra. Suppose that $x\in M$ with $\|x\|_{\infty}\leq 1.$ Suppose that  $x^{*}x$ has no nonzero fixed vectors when acting on $L^{2}(M,\tau).$ Then
\[\|x^{n}\xi\|_{2}\to_{n\to\infty} 0\]
for every $\xi\in L^{2}(M,\tau).$

\end{prop}

\begin{proof}
Note that $\|x^{n}\|_{\infty}\leq 1$ for every natural number $n.$ This uniform estimate implies that
\[\{\xi:\|x^{n}\xi\|_{2}\to_{n\to\infty}0\}\]
is a closed, linear subspace of $L^{2}(M,\tau).$ So it suffices to show that it contains $\{\hat{y}:y\in M\}.$ By the estimate
\[\|x^{n}\hat{y}\|_{2}=\|x^{n}y\|_{2}\leq \|x^{n}\|_{2}\|y\|_{\infty},\]
it suffices to show that $\|x^{n}\|_{2}\to_{n\to\infty}0.$
By repeated applications of the noncommutative H\o lder inequality we have
\[\|x^{n}\|_{2}\leq \|x\|_{2n}^{n}=\tau(|x|^{2n})^{1/2}=\left(\int t^{2n}\,d\mu_{|x|}(t)\right)^{1/2}.\]
The assumption that $x^{*}x$ has no nonzero fixed vectors when acting on $L^{2}(M,\tau)$ means that $1-x^{*}x$ is injective.  Since $\|x\|_{\infty}\leq 1,$ we know that $0\leq x^{*}x\leq 1.$ So $|1-x^{*}x|=1-x^{*}x.$ Hence by Proposition \ref{P:characterizing injectivity appendix}, we know
$1_{\{1\}}(x^{*}x)=1_{\{0\}}(1-x^{*}x)=0.$ Thus $\mu_{|x|}(\{1\})=0.$ Since $\|x\|_{\infty}\leq 1$ we have that $\mu_{|x|}([0,1])=1$ and so $\mu_{|x|}([0,1))=1.$ Thus
\[\int t^{2n}\,d\mu_{|x|}(t)=\int_{[0,1)}t^{2n}\,d\mu_{|x|}(t)\to_{n\to\infty}0,\]
by the dominated convergence theorem.

\end{proof}

As an application, we deduce Lemma \ref{L:sot convergence}.

\begin{proof}[Proof of Lemma \ref{L:sot convergence}]
Setting $M=L(G)$, and $\tau=\ip{\cdot \delta_{1},\delta_{1}}$ we have a canonical isomorphism $L^{2}(M,\tau)\cong \ell^{2}(G)$ which is equivariant with respect to the actions of $M$ and sends $\delta_{1}$ to $\hat{1}.$ So by Proposition \ref{P:more general sot convergence}, it suffices to show that $1-\lambda(x)^{*}\lambda(x)$ has no nonzero fixed vectors when acting on $\ell^{2}(G).$ Let $\nu=\widehat{x^{*}x}\in \Prob(G),$ and $S=\supp(\widehat{x}).$ Then $\supp(\nu)=S^{-1}S,$ and for all $\xi\in \ell^{2}(G)$ we have:
\[\sum_{g\in S^{-1}S}\|\lambda(g)\xi-\xi\|_{2}^{2}\nu(g)=2\|\xi\|_{2}^{2}-2\sum_{g\in S^{-1}S}\nu(g)\Re(\ip{\lambda(g)\xi,\xi}).\]
The fact that $\lambda(x)^{*}\lambda(x)$ is self-adjoint implies that $\nu(g)=\nu(g^{-1})$ for all $g\in G$ and from this it follows that
\[\sum_{g\in S^{-1}S}\|\lambda(g)\xi-\xi\|_{2}^{2}=2\|\xi\|_{2}^{2}-2\sum_{g\in S^{-1}S}\nu(g)\ip{\lambda(g)\xi,\xi}=2\|\xi\|_{2}^{2}-2\ip{\lambda(x)^{*}\lambda(x)\xi,\xi}.\]
Suppose that $\xi\in \ell^{2}(G)$ is fixed by $\lambda(x)^{*}\lambda(x).$ Then by the above equation and the fact that $\supp(\nu)=S^{-1}S,$ it follows that $\lambda(g)\xi=\xi$ for all $g\in S^{-1}S.$ Hence $\lambda(g)\xi=\xi$ for all $g\in \ip{S^{-1}S}.$ Since $\ip{S^{-1}S}$ is infinite and $\xi\in \ell^{2}(G),$ it follows that $\xi=0.$

\end{proof}

\subsection{General Results on $L^{2}$ formal inverses in von Neumann algebras and the Proof of Lemma \ref{C:reductoin to positive case}}\label{S:general L2 inverses}

In this section, we prove Lemma \ref{C:reductoin to positive case}. In fact, we will show that a version of Lemma \ref{C:reductoin to positive case} holds in an arbitrary tracial von Neumann algebra. To do this, we first generalize the notion of $\ell^{2}$ formal inverses from the case of the group ring to a general tracial von Neumann algebras.
\begin{defn}
Let $(M,\tau)$ be a tracial von Neumann algebra, and $x\in M.$ We say that $x$ has an \emph{$L^{2}$ formal inverse} if there is a $\xi\in L^{2}(M,\tau)$ with $x\xi=\hat{1}.$
\end{defn}

We recall (see \cite[VIII.3.11]{Conway}) the polar decomposition. Let $\mathcal{H}$ be a Hilbert space. We say that a $U\in B(\mathcal{H})$ is a \emph{partial isometry} if  $\|U\xi\|=\|\xi\|$ for all $\xi$ which are orthogonal to the kernel of $U.$ We leave it as an exercise to verify that if $U$ is a partial isometry, then $U^{*}U$ is the orthogonal projection onto the orthogonal complement of the kernel of $U,$ and $UU^{*}$ is the orthogonal projection onto to the image of $U.$ We call $U^{*}U$ the \emph{source projection} of $U$ and $UU^{*}$ the \emph{range projection} of $U.$
If $T\in B(\mathcal{H})$ then we can write $T=U|T|$ where $U\in \mathcal{B}(\mathcal{H})$ is a partial isometry whose source projection is the orthogonal projection onto the kernel of $T$ and whose range projection is the orthogonal projection onto the closure of the image of $T.$ Moreover, this decomposition is unique: if $W$ is another partial isometry whose source projection is the orthogonal projection onto the kernel of $T$ and whose range projection is the orthogonal projection onto the closure of the image of $T,$ and if $P$ is any positive operator with $T=WP$ then necessarily $W=U,$ and $P=|T|.$ See \cite[VIII.3.11]{Conway} for more details.

\begin{lem}\label{L:integral sublevel char app}
Let $(M,\tau)$ be tracial von Neumann algebra and $y\in M.$ Then the following are equivalent:
\begin{itemize}
    \item $y$ has an $L^{2}$ formal inverse,
    \item $y$ is injective and $\int_{(0,\infty)}t^{-2}\,d\mu_{|y|}(t)<\infty.$
\end{itemize}
Moreover, if $y$ has an $L^{2}$ formal inverse $\xi$, then
\[\xi=L^{2}-\lim_{\varepsilon\to 0}1_{(\varepsilon,\infty)}(|y|)|y|^{-1}u^{*}\hat{1}\]
where $y=u|y|$ is the polar decomposition of $y.$
\end{lem}

We remark that  the notation $1_{(\varepsilon,\infty)}(|y|)|y|^{-1}$ is mildly abusive. It should be interpreted as $\phi_{\varepsilon}(|y|)$ , where $\phi_{\varepsilon}\colon [0,\infty)\to [0,\infty)$ is the \emph{bounded} Borel function  given by $\phi_{\varepsilon}(t)=1_{(\varepsilon,\infty)}(t)t^{-1}.$ We are not requiring that $|y|$ be invertible as a bounded operator to make sense of $1_{(\varepsilon,\infty)}(|y|)|y|^{-1}.$

\begin{proof}
Throughout the proof, let $y=u|y|$ be the polar decomposition to $y.$ For $\varepsilon>0,$ let  $a_{\varepsilon}=1_{(\varepsilon,\infty)}(|y|)|y|^{-1}u^{*}$.

First, suppose that $y$ has an $L^{2}$ formal inverse $\xi.$ By the same arguments as in \cite[Proposition 2.2]{MeWE}, this implies that $y$ is injective as a bounded operator on $L^{2}(M,\tau).$ Let $y=u|y|$ is the Polar decomposition to $y.$  We compute:
\[\|\xi-a_{\varepsilon}\widehat{1}\|_{2}=\|(1-a_{\varepsilon}y)\xi\|_{2}=\|1_{[0,\varepsilon)}(|y|)\xi\|_{2}.\]
By the spectral theorem, this implies that
\[\lim_{\varepsilon\to 0}\|\xi-a_{\varepsilon}\widehat{1}\|_{2}=\|1_{\{0\}}(|y|)\xi\|_{2}.\]
 As $y$ is injective, Proposition \ref{P:characterizing injectivity appendix} implies that $1_{\{0\}}(|y|)=0.$ So
\[\xi=L^{2}-\lim_{\varepsilon\to 0}a_{\varepsilon}\widehat{1}=L^{2}-\lim_{\varepsilon\to 0}1_{(\varepsilon,\infty)}(|y|)|y|^{-1}u^{*}\widehat{1}.\]
Thus
\[\|\xi\|_{2}=\lim_{\varepsilon\to 0}\tau(|a_{\varepsilon}|^{2})^{1/2}=\lim_{\varepsilon\to 0}\tau(u1_{(\varepsilon,\infty)}(|y|)|y|^{-2}u^{*})^{1/2}=\lim_{\varepsilon\to 0}\tau(1_{(\varepsilon,\infty)}(|y|)|y|^{-2}u^{*}u)^{1/2}.\]
Since $y=u|y|$ is the polar decomposition of $y,$ we know that $u^{*}u$ is the orthogonal projection onto the orthogonal complement of the kernel of $y.$ Since $y$ is injective, we deduce that $u^{*}u=1.$
So
\[\|\xi\|_{2}=\lim_{\varepsilon\to 0}\tau(1_{(\varepsilon,\infty)}(|y|)|y|^{-2})^{1/2}=\lim_{\varepsilon\to 0}\left(\int_{(\varepsilon,\infty)}t^{-2}\,d\mu_{|y|}(t)\right)^{1/2}=\left(\int_{(0,\infty)}t^{-2}\,d\mu_{|y|}(t)\right)^{1/2},\]
the last step following from the monotone convergence theorem. Thus $\int_{(0,\infty)}t^{-2}\,d\mu_{|y|}(t)<\infty$.

Conversely, suppose that $y$ is injective and that $\int_{(0,\infty)}t^{-2}\,d\mu_{|y|}(t)<\infty.$ Since $y$ is injective, we may argue as in the preceding paragraph to see that $u^{*}u=1.$ Since $M$ has a faithful, normal tracial state, it follows by \cite[Theorem V.2.4]{Taka} that $uu^{*}=1.$ So $u$ is a unitary. Let $\xi_{\varepsilon}=a_{\varepsilon}\hat{1}.$ For $0<\delta<\varepsilon$ we have
\[\|\xi_{\delta}-\xi_{\varepsilon}\|_{2}=\|1_{(\delta,\varespilon)}(|y|)|y|^{-1}\widehat{u^{*}}\|_{2}=
\|1_{(\delta,\varespilon)}(|y|)|y|^{-1}\hat{1}u^{*}\|_{2}=\|1_{(\delta,\varespilon)}(|y|)|y|^{-1}\|_{2},\]
the last equality following as $u$ is unitary. So
\[\|\xi_{\delta}-\xi_{\varepsilon}\|_{2}^{2}=\tau(|1_{(\delta,\varepsilon)}(|y|)|y|^{-1}u^{*}|^{2})=\tau(u|1_{(\delta,\varepsilon)}(|y|)|y|^{-2}u^{*})=\tau(1_{(\delta,\varepsilon)}(|y|)|y|^{-2}),\]
the last line following by the fact that $\tau$ is tracial and $u$ is a unitary.
Since $\int_{[0,\infty)}t^{-2}\,d\mu_{|y|}(t)<\infty,$ the dominated convergence theorem shows that $\xi_{\varespilon}$ is a Cauchy net. Thus, by completeness of $L^{2}(M,\tau),$ we may define $\xi=\lim_{\varepsilon\to 0}\xi_{\varepsilon}$ where the limit is taken in $\|\cdot\|_{2}.$ We then have that
\[y\xi=\lim_{\varepsilon\to 0}u1_{(\varepsilon,\infty)}(|y|)u^{*}\hat{1}.\]
By the spectral theorem, we know $y\xi=u1_{(0,\infty)}(|y|)u^{*}\hat{1}.$ Injectivity of $y$ and functional calculus imply that $1_{(0,\infty)}(|y|)=1-1_{\{0\}}(|y|)=1.$ So
\[y\xi=uu^{*}\hat{1}=\hat{1},\]
the last step following as $u$ is unitary.

\end{proof}

\begin{cor}\label{C:sublevel set characterization}
Let $(M,\tau)$ be tracial von Neumann algebra and suppose that $y\in M$ is injective as an operator on $L^{2}(M,\tau).$
Then $y$ has an $L^{2}$ formal inverse if and only if
\[\int_{0}^{\infty}\lambda \mu_{|y|}\left(\left(0,\frac{1}{\lambda}\right)\right)\,d\lambda<\infty.\]
\end{cor}

\begin{proof}
By direct computation
\[\int_{(0,\infty)}t^{-2}\,d\mu_{|y|}(t)=2\int_{(0,\infty)}\int_{0}^{1/t}\lambda\,d\lambda\,d\mu_{|y|}(t)=2\int_{0}^{\infty}
\lambda\mu_{|y|}\left(\left(0,\frac{1}{\lambda}\right)\right)\,d\mu_{|y|}(t).\]
\end{proof}
Note that, by definition, $\mu_{|y|}\left(\left(0,\frac{1}{\lambda}\right)\right)=\tau(1_{(0,\frac{1}{\lambda})}(|y|)).$
Recall that our overall goal is to show that if $x\in M$ and $\|x\|_{\infty}\leq 1,$ and if $1-\Re(x)$ has an $L^{2}$ formal inverse, then so does $1-x$. The above corollary will facilitate this goal because it turns out that we may effectively estimate $\tau(1_{(0,\frac{1}{\lambda})}(|1-x|))$ in terms of $\tau(1_{(0,\frac{1}{\lambda})}(1-\Re(x))).$ The main tool that helps us do this is the following, which is a rephrasing of \cite[Lemma IX.2.6]{TakesakiII}.

\begin{lem} Let $(M,\tau)$ be a tracial von Neumann algebra and $p,q$ orthogonal projections in $M.$ If $p(L^{2}(M,\tau))\cap (1-q)(L^{2}(M,\tau)))=\{0\},$ then $\tau(p)\leq \tau(q).$
\end{lem}

The following are the main estimates we will use in our proof of Lemma \ref{C:reductoin to positive case}.

\begin{lem}\label{L:nc subordination}
Let $(M,\tau)$ be a tracial von Neumann algebra.
\begin{enumerate}[(i)]
\item  Suppose that $a,b\in M$ with $0\leq a\leq b.$
Then $\mu_{b}([0,t))\leq \mu_{a}([0,t))$ for all $t>0.$
\label{I:monotonicityL2formalinverse}
\item Suppose that $x\in M$ and that $\|x\|_{\infty}\leq 1.$ Then $\mu_{|1-x|}([0,t))\leq \mu_{1-\Re(x)}([0,t))$ for all $t>0.$ \label{I:NC reverse triangle inequality}
\end{enumerate}
\end{lem}

\begin{proof}

(\ref{I:monotonicityL2formalinverse}): Set $p=1_{[0,t)}(b)$,$q=1_{[0,t)}(a).$ Suppose $p(L^{2}(M,\tau))\cap (1-q)(L^{2}(M,\tau))\ne \varnothing.$ Let $\xi\in p(L^{2}(M,\tau))\cap (1-q)(L^{2}(M,\tau))$ with $\|\xi\|_{2}=1.$ By the spectral theorem, we may find compactly supported measures $\nu,\eta\in \Prob([0,\infty))$ with
\[\ip{f(a)\xi,\xi}=\int f\,d\nu\,\mbox{ and } \ip{f(b)\xi,\xi}=\int f\,d\eta\]
for all bounded, Borel functions $f\colon [0,\infty)\to \C.$
Since $p\xi=\xi,$ it follows that
\[\ip{b\xi,\xi}=\ip{b1_{[0,t)}(b)\xi,\xi}=\int_{[0,t)}s\,d\eta(s)<t.\]
Since $1-q=1_{[t,\infty)}(a)$ and $(1-q)\xi=\xi,$ we know that
\[\ip{a\xi,\xi}=\ip{a1_{[t,\infty)}(a)\xi,\xi}=\int_{[t,\infty)}s\,d\eta(s)\geq t\eta([t,\infty))=t\ip{(1-q)\xi,\xi}=t.\]
So we have
\[t\leq \ip{a\xi,\xi}\leq \ip{b\xi,\xi}<t,\]
a contradiction.

(\ref{I:NC reverse triangle inequality}): Set $p=1_{[0,t)}(|1-x|),$ and $q=1_{[0,t)}(1-\Re(x)).$ Suppose $p(L^{2}(M,\tau))\cap (1-q)(L^{2}(M,\tau))\ne \{0\}.$ Then we may find a $\xi\in p(L^{2}(M,\tau))\cap (1-q)(L^{2}(M,\tau))$ with $\|\xi\|_{2}=1.$ By the spectral theorem, we may choose  compactly supported $\eta,\nu\in \Prob([0,\infty))$ with
\[\ip{f(|1-x|)\xi,\xi}=\int f\,d\nu,\]
\[\ip{f(\Re(x))\xi,\xi}=\int f\,d\eta,\]
for all bounded, Borel $f\colon [0,\infty)\to \C.$ Since $p\xi=\xi,$ it follows that
\[\|(1-x)\xi\|_{2}^{2}=\ip{|1-x|^{2}\xi,\xi}=\ip{|1-x|^{2}p\xi,\xi}=\int_{[0,t)}s^{2}\,d\nu(s)<t^{2}.\]
On the other hand,
\[\|(1-x)\xi\|_{2}\geq \Re(\ip{(1-x)\xi,\xi})=\ip{\Re(1-x)\xi,\xi}=\ip{\Re(1-x)(1-q)\xi,\xi},\]
the last equality following as $(1-q)\xi=\xi.$ By definition of $q$
\[\|(1-x)\xi\|_{2}^{2}\geq \int_{[t,\infty]}s^{2}\,d\eta(s)\geq t^{2}\eta([t,\infty))=t^{2}\ip{(1-q)\xi,\xi}=t^{2},\]
where in the last equality we again use that $(1-q)\xi=\xi.$
So we have shown that
\[t<\|(1-x)\xi\|\leq t,\]
 a contradiction.
\end{proof}

We now prove the main result of this section.

\begin{lem}\label{lem:general reduction to sa case appendix}
Let $(M,\tau)$ be a tracial von Neumann algebra and $x\in M$ with $\|x\|_{\infty}\leq 1.$ If $1-\Re(x)$ has an $L^{2}$ formal inverse, then so does $1-x.$
\end{lem}

\begin{proof}

Assume $1-\Re(x)$ has an $L^{2}$ formal inverse. Recall that this implies that $1-\Re(x)$ is injective. We first show that $1-x$ is injective. Suppose that $\xi\in L^{2}(M,\tau)$ and $x\xi=\xi.$ Then
\[\|x^{*}\xi-\xi\|_{2}^{2}=\|x^{*}\xi\|_{2}^{2}+\|\xi\|_{2}^{2}-2\Re(\ip{x^{*}\xi,\xi})=\|x^{*}\xi\|_{2}^{2}+\|\xi\|_{2}^{2}-2\Re(\ip{\xi,x\xi})=\|x^{*}\xi\|_{2}^{2}-\|\xi\|_{2}^{2}\leq 0,\]
the last step following as $\|x^{*}\|_{\infty}=\|x\|_{\infty}\leq 1.$ So $x^{*}\xi=\xi$, and thus $\Re(x)\xi=\xi.$ By injectivity of $1-\Re(x),$ it follows that $\xi=0$. So $1-x$ is injective.  It thus remains to show that $\int_{0}^{\infty}\lambda\mu_{|1-x|}\left(\left(0,\frac{1}{\lambda}\right)\right)\,d\lambda<\infty.$


By Lemma \ref{L:nc subordination} (\ref{I:NC reverse triangle inequality}), we know
\[\mu_{|1-x|}\left(\left[0,\frac{1}{\lambda}\right)\right)\leq \mu_{1-\Re(x)}\left(\left[0,\frac{1}{\lambda}\right)\right)=\mu_{1-\Re(x)}\left(\left(0,\frac{1}{\lambda}\right)\right)\]
for every $\lambda>0.$ Since we just saw that $1-x$ is injective, we know by Proposition \ref{P:characterizing injectivity appendix} that $\mu_{|1-x|}(\{0\})=0.$  So
\[\mu_{|1-x|}\left(\left(0,\frac{1}{\lambda}\right)\right)\leq \mu_{1-\Re(x)}\left(\left(0,\frac{1}{\lambda}\right)\right).\]
Additionally, the fact that $\|x\|_{\infty}\leq 1$ implies that $\|\Re(x)\|_{\infty}\leq 1.$ So $-1\leq \Re(x)\leq 1,$ and thus $1-\Re(x)\geq 0.$ So $|1-\Re(x)|=1-\Re(x),$ and the Lemma now follows from the above estimate and Corollary \ref{C:sublevel set characterization}.
\end{proof}

We now prove Lemma \ref{C:reductoin to positive case}.

\begin{proof}[Proof of Lemma \ref{C:reductoin to positive case}]
Recall that if $M=L(G)$ and $\tau=\ip{\cdot \delta_{1},\delta_{1}},$ then we may have a canonical identification $L^{2}(M,\tau)\cong \ell^{2}(G)$ which sends $\hat{1}$ to $\delta_{1}.$ So it is enough to show that if $1-\left(\frac{\lambda(x)^{*}+\lambda(x)}{2}\right)$ has an $L^{2}$ formal inverse, then $1-\lambda(x)$ has an $L^{2}$ formal inverse.

By the triangle inequality, if $x\in \C(G)$ and $|\widehat{x}|\in \Prob(G),$ then $\|\lambda(x)\|\leq 1.$
So Lemma \ref{C:reductoin to positive case} now follows from Lemma \ref{lem:general reduction to sa case appendix}.

\end{proof}
\subsection{Fuglede-Kadison determinants of lopsided and semi-lopsided elements}\label{S:FKD}

We now recall the definition of Fuglede-Kadison determinants for elements in a tracial von Neumann algebra.

\begin{defn}
Let $(M,\tau)$ be a tracial von Neumann algebra, and let $x\in M.$ Define the Fuglede-Kadison determinant of $x$ by
\[\Det_{M}(x)=\exp\left(\int \log(t)\,d\mu_{|x|}(t)\right)\]
with the convention that $\exp(-\infty)=0.$
\end{defn}

 Let $(M,\tau)$ be a tracial von Neumann algebra. Given any $x\in M,$ by \cite{BrownFKD} there is a unique probability measure $\mu_{x}$ defined on the Borel subsets of the spectrum of $x,$ so that
\[\log\Det_{M}(\lambda-x)=\int \log|\lambda-z|\,d\mu_{x}(z)\]
for all $\lambda\in \C.$ In the case that $x$ is self-adjoint, this agrees with the measure $\mu_{x}$ defined earlier in this section.
We call $\mu_{x}$ the \emph{Brown measure} of $x,$ this measure is an analogue of the eigenvalue distribution of $x$ (it may be that the support of $\mu_{x}$ is a \emph{proper} subset of the spectrum of $x$). Brown showed in Theorem 3.10 of \cite{BrownFKD} that:
\[\tau(x^{k})=\int z^{k}\,d\mu_{x}(z)\]
for all $k\in \N.$

The following result is how we compute Fuglede-Kadison determinants for semi-lopsided elements. This result is surely well known, but we include the proof for completeness.
\begin{prop}
 Let $(M,\tau)$ be a tracial von Neumann algebra, and let $x\in M$ be a contraction (i.e. $\|x\|\leq 1$). Then:
\begin{enumerate}[(a)]
\item\label{I:limitingformula}
$\lim_{t\to 1}\Det_{M}(1-tx)=\Det_{M}(1-x).$
\item\label{I:sumformula} If $\sum_{k=1}^{\infty}\frac{\tau(x^{k})}{k}$ is conditionally convergent, then
\[\Det_{M}(1-x)=\left|\exp\left(-\sum_{k=1}^{\infty}\frac{\tau(x^{k})}{k}\right)\right|.\]
\end{enumerate}
\end{prop}

\begin{proof}

(\ref{I:limitingformula}): Let $\mu_{x}$ be the Brown measure of $x.$ By definition,
\[\int \log|1-tz|\,d\mu_{x}(z)=\log\Det_{M}(1-tx),\]
and $\mu_{x}$ is a measure supported on $\Bbb D=\{z\in \C:|z|\leq 1\}.$ For $t\in (0,\infty),$
\[\Det_{M}(1-tx)=t\Det_{M}(t^{-1}-x).\]
Observe that, for all $z\in \Bbb D,$  and all $t,s\in (0,1)$ with $t<s<1$ we have that
\[-\log|t^{-1}-z|\leq -\log|s^{-1}-z|,\]
and $-\log|t^{-1}-z|\geq -\log(3)$ for all $t\in (\frac{1}{2},1).$ So it follows from the monotone convergence theorem that
\[\lim_{t\to 1}\log \Det_{M}(1-tx)=\lim_{t\to 1}\log(t)+\int \log|t^{-1}-z|\,d\mu_{x}(z)=\int \log|1-z|\,d\mu_{x}(z)=\log\Det_{M}(1-x).\]

(\ref{I:sumformula}):
It suffices to show that
\[\log \Det_{M}(1-x)=\Re\left(-\sum_{k=1}^{\infty}\frac{\tau(x^{k})}{k}\right).\]
Let $\log(z)$ denote the branch of the complex logarithm defined on the right-half plane and which has $\log(1)=0.$
By part $(a)$ we have that
\[\log\Det_{M}(1-x)=\lim_{t\to 1}\log\Det_{M}(1-tx)=\lim_{t\to 1}\int \log|1-tz|\,d\mu_{x}(z)=\lim_{t\to 1}\Re\left(\int \log(1-tz)\,d\mu_{x}(z)\right).\]
For $0<t<1,$ the sum $\sum_{k=1}^{\infty}\frac{t^{k}z^{k}}{k}$ converges uniformly on $\Bbb D$ to $-\log(1-tz).$ Hence, for $0<t<1:$
\[\int\log(1-tz)\,d\mu_{x}(z)=-\sum_{k=1}^{\infty}\int \frac{t^{k}z^{k}}{k}\,d\mu_{x}(z)=-\sum_{k=1}^{\infty}\frac{t^{k}\tau(x^{k})}{k},\]
the last equality following by Theorem 3.10 of \cite{BrownFKD}. By Abel's theorem
\[\lim_{t\to 1}\sum_{k=1}^{\infty}\frac{t^{k}\tau(x^{k})}{k}=\sum_{k=1}^{\infty}\frac{\tau(x^{k})}{k},\]
since $\sum_{k=1}^{\infty}\frac{\tau(x^{k})}{k}$ is conditionally convergent. Thus we have that
\[\log\Det_{M}(1-x)=\Re\left(-\sum_{k=1}^{\infty}\frac{\tau(x^{k})}{k}\right).\]

\end{proof}

\begin{cor}\label{C:FKD Calculation}
Let $G$ be a countable, discrete, group and suppose that $G$ has a left-invariant partial order $\preceq.$ Let $S\subseteq \{g\in G:g\succ 1\},$ and let $(a_{s})_{s\in S}$ be complex numbers. Suppose that $a\in \C\setminus\{0\}$ and that $\sum_{s\in S}|a_{s}|\leq |a|.$ Finally, set $f=a+\sum_{s\in S}a_{s}s.$ Then
\[\log\Det_{L(G)}(f)=\log|a|.\]
\end{cor}

\begin{proof}
Set $x=-\frac{1}{a}\sum_{s\in S}a_{s}s,$ then $\|x\|\leq 1.$
Since $f=a(1-x),$ we have:
\[\log\Det_{L(G)}(f)=\log|a|+\log\Det_{L(G)}(1-x)=\log|a|+\Re\left(-\sum_{k=1}^{\infty}\frac{\tau(x^{k})}{k}\right).\]
Since $S\subseteq\{g\in G:g\succ 1\},$ it is not hard to see that $\tau(x^{k})=0$ for every integer $k\geq 1.$ This completes the proof.

\end{proof}

\section{On $\ell^{2}$ formal inverses}

\subsection{ $\ell^{p}$ formal inverses of balanced elements}\label{S:inverses}

We start by addressing some of the invertibility conditions in the paper and show that in may cases they are optimal. The invertibility conditions on $f\in \Z(G)$ that occur in this paper and previous other works (in increasing order of generality) are typically the following:
\begin{itemize}
    \item $f$ has an $\ell^{1}$ formal inverse, \cite{BowenEntropy, Den, DenSchmidt, LSVHomoc},
    \item $f$ is invertible in the full $C^{*}$-algebra of $G,$ \cite{KLi},
    \item $\lambda(f)$ is invertible (\cite{Li, Me7, Me13}),
    \item $f$ has an $\ell^{2}$ formal inverse (\cite{MeWE}).
\end{itemize}

Since they are the ones relevant to our paper, we  discuss when the first, third and fourth  of these invertibility hypotheses for well balanced $f$ occur in the following proposition. Compare \cite[Appendix A]{Li} for a more involved discussion in the amenable case of how and when the first and second hypotheses differ.

\begin{prop}
Let $G$ be a countable, discrete group and $f\in \Z(G)$ be semi-lopsided. Let $H=\ip{\supp(\widehat{f})}.$ Then
\begin{enumerate}[(a)]
\item  if $f$ is well-balanced, it does not have an $\ell^{1}$ formal inverse, \label{I:easy stuff inverses again}
\item if $f$ is well-balanced, then $\lambda(f)$ is invertible if and only if $G$ is amenable, \label{I:amenable charcterize inverses}
\item if $G$ is nonamenable, then $\lambda(f)$ is invertible. \label{I: this is not hard either}
\end{enumerate}

\end{prop}

\begin{proof}

(\ref{I:easy stuff inverses again}): Define $t\colon \ell^{1}(G)\to \C$ by $t(\xi)=\sum_{g\in G}\xi(g).$ Direct computations show that $t(\xi*\eta)=t(\xi)t(\eta)$ for all $\xi,\eta\in \ell^{1}(G)$ where $*$ is convolution. So $t(f\xi)=t(\widehat{f})t(\xi)=0$ for all $\xi\in \ell^{1}(G).$ But then obviously there is no $\xi\in\ell^{1}(G)$ with $f\xi=\delta_{1}.$

(\ref{I:amenable charcterize inverses}): Write $f=m(1-x)$ where $\widehat{x}\in \Prob(G).$ First suppose that $H$ is amenable. Then there is a sequence $(\xi_{n})_{n\in\N}$ in $\ell^{2}(G)$ with $\|\xi_{n}\|_{2}=1$ and $\|\lambda(h)\xi_{n}-\xi_{n}\|_{2}\to_{n\to\infty}0$ by \cite[Appendix G]{BHV}. Since $\widehat{x}\in \Prob(G),$ we thus have that $\|\xi_{n}-x\xi_{n}\|_{2}\to 0.$ Hence, we have that $\|(1-x)\xi_{n}\|_{2}\to 0.$ By the open mapping theorem, if $1-\lambda(x)$ were invertible we would have that there is a constant $C>0$ so that $\|(1-\lambda(x))\zeta\|_{2}\geq C\|\zeta\|_{2}$ for all $\zeta\in \ell^{2}(G).$ Since $\|(1-x)\xi_{n}\|_{2}\to_{n\to\infty}0,$ and $\|\xi_{n}\|_{2}=1,$ we must have that $1-\lambda(x)$ is not invertible. So $\lambda(f)$ is not invertible.

Conversely, suppose that $H$ is not amenable. Then, by \cite[Appendix G]{BHV} we have that $\|\lambda(x)\|_{B(\ell^{2}(G))}<1.$ So $1-\lambda(x)$ is invertible and it inverse is given by $\sum_{n}\lambda(x)^{n}.$

(\ref{I: this is not hard either}): If $f$ is lopsided, this is obvious. So we may assume that $f=m(1-x)$ where $|\widehat{x}|\in \Prob(G).$ Let $y\in \R(G)$ be so that $\widehat{y}=|\widehat{x}|. $ Then, for all $\xi\in \ell^{2}(G),$ we have that $|x\xi|\leq y|\xi|,$ so $\|x\xi\|_{2}\leq \|y|\xi|\|_{2}\leq \|\lambda(y)\|_{B(\ell^{2}(G))}\|\xi\|_{2}.$ So $\|\lambda(x)\|_{B(\ell^{2}(G))}\leq \|\lambda(y)\|_{B(\ell^{2}(G))}.$ As in (\ref{I:amenable charcterize inverses}), we know that $\|\lambda(y)\|_{B(\ell^{2}(G))}<1.$ So $\|\lambda(x)\|_{B(\ell^{2}(G))}<1$ as well, and as in (\ref{I:amenable charcterize inverses}) this implies that $1-\lambda(x)$ is invertible. So $\lambda(f)$ is invertible.

\end{proof}

It is easy to see that $\lambda(f)$ being invertible is equivalent to $f$ being invertible in the group von Neumann algebra (as used and discussed in \cite{Li, Me7, Me13}).
By our work in Section \ref{S:growth} we know that, in the setup of the above Proposition, that if $\ip{a^{-1}b:a,b\in \supp(\widehat{f})\setminus\{1\}}$ is infinite, and if  $H$ either has superpolynomial growth or polynomial growth of degree at least $5,$ then $f$ has an $\ell^{2}$ formal inverse. If $f$ is well-balanced, and $H$ is amenable, then $\lambda(f)$ is not invertible. So we have many examples of $f$ which have an $\ell^{2}$ formal inverse, but are not invertible in the group von Neumann algebra.


\begin{thebibliography}{10}

\bibitem{RWAlex}
G.~K. Alexopoulos.
\newblock Random walks on discrete groups of polynomial volume growth.
\newblock {\em Ann. Probab.}, 30(2):723--801, 2002.

\bibitem{AlpMeyRyu}
A.~Alpeev, T.~Meyerovitch, and S.~Ryu.
\newblock Predictability, topological entropy and invariant random orders.
\newblock {\em arXiv:1812.10833}.

\bibitem{AustinGeo}
T.~Austin.
\newblock The geometry of model spaces for probability-preserving actions of
  sofic groups.
\newblock {\em to appear in Anal. Geom. Metr. Spaces}.

\bibitem{BHV}
B.~Bekka, P.~de~la Harpe, and A.~Valette.
\newblock {\em Kazhdan's property ({T})}, volume~11 of {\em New Mathematical
  Monographs}.
\newblock Cambridge University Press, Cambridge, 2008.

\bibitem{BowenExamples}
L.~Bowen.
\newblock Examples in the entropy theory of countable group actions.
\newblock {\em Ergodic Theory Dynam. Systems}.
\newblock to appear.

\bibitem{BowenGeneric}
L.~Bowen.
\newblock Zero entropy is generic.
\newblock {\em arXiv:1603.02621}.

\bibitem{Bow}
L.~Bowen.
\newblock Measure conjugacy invariants for actions of countable sofic groups.
\newblock {\em J. Amer. Math. Soc.}, 23(1):217--245, 2010.

\bibitem{BowenEntropy}
L.~Bowen.
\newblock Entropy for expansive algebraic actions of residually finite groups.
\newblock {\em Ergodic Theory Dynam. Systems}, 31(3):703--718, 2011.

\bibitem{BowenOrn}
L.~Bowen.
\newblock Every countably infinite group is almost {O}rnstein.
\newblock In {\em Dynamical systems and group actions}, volume 567 of {\em
  Contemp. Math}, pages 67--78. Amer. Math. Soc, Providence, RI, 2012.

\bibitem{BowenLi}
L.~Bowen and H.~Li.
\newblock Harmonic models and spanning forests of residually finite groups.
\newblock {\em J. Funct. Anal.}, 263(7):1769--1808, 2012.

\bibitem{BrownFKD}
L.~G. Brown.
\newblock Lidski\u\i 's theorem in the type {${\rm II}$} case.
\newblock In {\em Geometric methods in operator algebras ({K}yoto, 1983)},
  volume 123 of {\em Pitman Res. Notes Math. Ser.}, pages 1--35. Longman Sci.
  Tech., Harlow, 1986.

\bibitem{NotesThompson}
J.~W. Cannon, W.~J. Floyd, and W.~R. Parry.
\newblock Introductory notes on {R}ichard {T}hompson's groups.
\newblock {\em Enseign. Math. (2)}, 42(3-4):215--256, 1996.

\bibitem{SolidErg}
I.~Chifan and A.~Ioana.
\newblock Ergodic subequivalence relations induced by a {B}ernoulli action.
\newblock {\em Geom. Funct. Anal.}, 20(1):53--67, 2010.

\bibitem{Conway}
J.~B. Conway.
\newblock {\em A course in functional analysis}, volume~96 of {\em Graduate
  Texts in Mathematics}.
\newblock Springer-Verlag, New York, second edition, 1990.

\bibitem{silva2018lecture}
R.~C. da~Silva.
\newblock Lecture notes on noncommutative lp-spaces, 2018.

\bibitem{Dehornoy}
P.~Dehornoy.
\newblock Braid groups and left distributive operations.
\newblock {\em Trans. Amer. Math. Soc.}, 345(1):115--150, 1994.

\bibitem{Den}
C.~Deninger.
\newblock Fuglede-{K}adison determinants and entropy for actions of discrete
  amenable groups.
\newblock {\em J. Amer. Math. Soc.}, 19(3):737--758, 2006.

\bibitem{DenSchmidt}
C.~Deninger and K.~Schmidt.
\newblock Expansive algebraic actions of discrete residually finite amenable
  groups and their entropy.
\newblock {\em Ergodic Theory Dynam. Systems}, 27:769--786, 2007.

\bibitem{DixmierLp}
J.~Dixmier.
\newblock Formes lin\'{e}aires sur un anneau d'op\'{e}rateurs.
\newblock {\em Bull. Soc. Math. France}, 81:9--39, 1953.

\bibitem{GabSewPopaFactor}
D.~Gaboriau and B.~Seward.
\newblock Factors of {B}ernoulli and treeability.
\newblock in progress.

\bibitem{GhysOrder}
E.~Ghys.
\newblock Groups acting on the circle.
\newblock {\em Enseign. Math. (2)}, 47(3-4):329--407, 2001.

\bibitem{GrigorIntGrowth3}
R.~I. Grigorchuk.
\newblock Degrees of growth of {$p$}-groups and torsion-free groups.
\newblock {\em Mat. Sb. (N.S.)}, 126(168)(2):194--214, 286, 1985.

\bibitem{GRMakiOrder}
R.~I. Grigorchuk and A.~Mach\'{i}.
\newblock On a group of intermediate growth that acts on a line by
  homeomorphisms.
\newblock {\em Mat. Zametki}, 53(2):46--63, 1993.

\bibitem{MeMaxMinWC}
B.~Hayes.
\newblock Max-min theorems for weak containment, square summable homoclinic
  points, and completely positive entropy.
\newblock {\em arXiv:1902.06600}.

\bibitem{Me13}
B.~Hayes.
\newblock Relative entropy and the {P}insker product formula for sofic groups.
\newblock {\em Groups Geom. Dyn.}
\newblock to appear.

\bibitem{Me5}
B.~Hayes.
\newblock Fuglede--{K}adison determinants and sofic entropy.
\newblock {\em Geom. Funct. Anal.}, 26(2):520--606, 2016.

\bibitem{Me7}
B.~Hayes.
\newblock Independence tuples and {D}eninger's problem.
\newblock {\em Groups Geom. Dyn.}, 11(1):245--289, 2017.

\bibitem{MeWE}
B.~Hayes.
\newblock Weak equivalence to {B}ernoulli shifts for some algebraic actions.
\newblock {\em Proc. Amer. Math. Soc.}, 147(5):2021--2032, 2019.

\bibitem{HydeLodLeftOrd}
J.~Hyde and Y.~Lodha.
\newblock Finitely generated infinite simple groups of homeomorphisms of the
  real line.
\newblock {\em Invent. Math.}
\newblock to appear.

\bibitem{KatzErgAut}
Y.~Katznelson.
\newblock Ergodic automorphisms of {$T^{n}$} are {B}ernoulli shifts.
\newblock {\em Israel J. Math.}, 10:186--195, 1971.

\bibitem{KerrCPE}
D.~Kerr.
\newblock Bernoulli actions of sofic groups have completely positive entropy.
\newblock {\em Israel J. Math.}, 202(1):461--474, 2014.

\bibitem{KLi}
D.~Kerr and H.~Li.
\newblock Entropy and the variational principle for actions of sofic groups.
\newblock {\em Invent. Math.}, 186(3):501--558, 2011.

\bibitem{KimKobLodLeftOrd}
S.-h. Kim, T.~Koberda, and Y.~Lodha.
\newblock Chain groups of homeomorphisms of the interval.
\newblock {\em Ann. Sci. Ec. Norm. Sup\'{e}r}.
\newblock to appear.

\bibitem{KitchensSchmidt}
B.~Kitchens and K.~Schmidt.
\newblock Automorphisms of compact groups.
\newblock {\em Ergodic Theory Dynam. Systems}, 9(4):691--735, 1989.

\bibitem{Li}
H.~Li.
\newblock Sofic mean dimension.
\newblock {\em Adv. Math.}, 244:570--604, 2014.

\bibitem{LiThom}
H.~Li and A.~Thom.
\newblock Entropy, determinants, and {$L^2$}-torsion.
\newblock {\em J. Amer. Math. Soc.}, 27(1):239--292, 2014.

\bibitem{LindSchmidtFreeBern}
D.~Lind and K.~Schmidt.
\newblock A bernoulli algebraic action of a free group.
\newblock preprint.

\bibitem{LindSchmidt1}
D.~Lind, K.~Schmidt, and E.~Verbitskiy.
\newblock Entropy and growth rate of periodic points of algebraic {$\Bbb
  Z^d$}-actions.
\newblock In {\em Dynamical numbers---interplay between dynamical systems and
  number theory}, volume 532 of {\em Contemp. Math.}, pages 195--211. Amer.
  Math. Soc., Providence, RI, 2010.

\bibitem{LSVHomoc}
D.~Lind, K.~Schmidt, and E.~Verbitskiy.
\newblock Homoclinic points, atoral polynomials, and periodic points of
  algebraic {$\Bbb Z^d$}-actions.
\newblock {\em Ergodic Theory Dynam. Systems}, 33(4):1060--1081, 2013.

\bibitem{LindSchmidt2}
D.~Lind, K.~Schmidt, and T.~Ward.
\newblock Mahler measure and entropy for commuting automorphisms of compact
  groups.
\newblock {\em Invent. Math.}, 101(3):593--629, 1990.

\bibitem{LindErgAut}
D.~A. Lind.
\newblock Ergodic automorphisms of the infinite torus are {B}ernoulli.
\newblock {\em Israel J. Math.}, 17:162--168, 1974.

\bibitem{MorrisLeftOrd}
D.~W. Morris.
\newblock Amenable groups that act on the line.
\newblock {\em Algebr. Geom. Topol.}, 6:2509--2518, 2006.

\bibitem{NavasOrderDiffeo}
A.~Navas.
\newblock Growth of groups and diffeomorphisms of the interval.
\newblock {\em Geom. Funct. Anal.}, 18(3):988--1028, 2008.

\bibitem{OrnClassify1}
D.~Ornstein.
\newblock Bernoulli shifts with the same entropy are isomorphic.
\newblock {\em Advances in Math.}, 4:337--352, 1970.

\bibitem{OrnClassify2}
D.~Ornstein.
\newblock Two {B}ernoulli shifts with infinite entropy are isomorphic.
\newblock {\em Advances in Math.}, 5:339--348 (1970), 1970.

\bibitem{OrnWeiss}
D.~S. Ornstein and B.~Weiss.
\newblock Entropy and isomorphism theorems for actions of amenable groups.
\newblock {\em J. Analyse Math.}, 48:1--141, 1987.

\bibitem{PopaCohomologyOE}
S.~Popa.
\newblock Some computations of 1-cohomology groups and construction of
  non-orbit-equivalent actions.
\newblock {\em J. Inst. Math. Jussieu}, 5(2):309--332, 2006.

\bibitem{PopaSasyk}
S.~Popa and R.~Sasyk.
\newblock On the cohomology of {B}ernoulli actions.
\newblock {\em Ergodic Theory Dynam. Systems}, 27(1):241--251, 2007.

\bibitem{OrderMCG}
C.~Rourke and B.~Wiest.
\newblock Order automatic mapping class groups.
\newblock {\em Pacific J. Math.}, 194(1):209--227, 2000.

\bibitem{RudolphSchmidtCPE}
D.~J. Rudolph and K.~Schmidt.
\newblock Almost block independence and {B}ernoullicity of {${\bf
  Z}^d$}-actions by automorphisms of compact abelian groups.
\newblock {\em Invent. Math.}, 120(3):455--488, 1995.

\bibitem{SchmidtSpectralGap}
K.~Schmidt.
\newblock Amenability, {K}azhdan's property {$T$}, strong ergodicity and
  invariant means for ergodic group-actions.
\newblock {\em Ergodic Theory Dynamical Systems}, 1(2):223--236, 1981.

\bibitem{Sandpiles}
K.~Schmidt and E.~Verbitskiy.
\newblock Abelian sandpiles and the harmonic model.
\newblock {\em Comm. Math. Phys.}, 292(3):721--759, 2009.

\bibitem{SewardOrn}
B.~Seward.
\newblock Bernoulli shifts with bases of equal entropy are isomorphic.
\newblock {\em arXiv:1805.08279}.

\bibitem{SewardSinai}
B.~Seward.
\newblock Positive entropy actions of countable groups factor onto {B}ernoulli
  shifts.
\newblock {\em arXiv:1804.05269}.

\bibitem{GeomOrderMCG}
H.~Short and B.~Wiest.
\newblock Orderings of mapping class groups after {T}hurston.
\newblock {\em Enseign. Math. (2)}, 46(3-4):279--312, 2000.

\bibitem{Stepin}
A.~M. Stepin.
\newblock Bernoulli shifts on groups.
\newblock {\em Dokl. Akad. Nauk SSSR}, 223(2):300--302, 1975.

\bibitem{Taka}
M.~Takesaki.
\newblock {\em Theory of operator algebras. {I}}, volume 124 of {\em
  Encyclopaedia of Mathematical Sciences}.
\newblock Springer-Verlag, Berlin, 2002.
\newblock Reprint of the first (1979) edition, Operator Algebras and
  Non-commutative Geometry, 5.

\bibitem{TakesakiII}
M.~Takesaki.
\newblock {\em Theory of operator algebras. {II}}, volume 125 of {\em
  Encyclopaedia of Mathematical Sciences}.
\newblock Springer-Verlag, Berlin, 2003.
\newblock Operator Algebras and Non-commutative Geometry, 6.

\bibitem{RobinAF}
R.~Tucker-Drob.
\newblock Mixing actions of countable groups are almost free.
\newblock {\em Proc. Amer. Math. Soc.}, 143(12):5227--5232, 2015.

\bibitem{VarRWGroups}
N.~T. Varopoulos.
\newblock Wiener-{H}opf theory and nonunimodular groups.
\newblock {\em J. Funct. Anal.}, 120(2):467--483, 1994.

\bibitem{OrderedFreeProducts}
A.~A. Vinogradov.
\newblock On the free product of ordered groups.
\newblock {\em Mat. Sbornik N.S.}, 25(67):163--168, 1949.

\end{thebibliography}

\end{document}